\documentclass[11pt,reqno,a4paper]{amsart}

\usepackage{etoolbox}
\usepackage{amsmath}
\usepackage{enumerate}
\usepackage{amssymb}
\usepackage{amscd}
\usepackage{amsthm}
\usepackage{amsfonts}
\usepackage{graphicx}
\usepackage[all,cmtip]{xy}

\usepackage{todonotes} 

\patchcmd{\subsection}{-.5em}{.5em}{}{}
\patchcmd{\subsubsection}{-.5em}{.5em}{}{}

\usepackage{enumitem}

\usepackage[T1]{fontenc}

\makeatother
\usepackage{hyperref}

\numberwithin{equation}{section}

\newcommand{\cA}{\mathcal{A}}

\newcommand{\cC}{\mathcal{C}}
\newcommand{\cD}{\mathcal{D}}

\newcommand{\cL}{\mathcal{L}}

\newcommand{\cN}{\mathcal{N}}

\newcommand{\cP}{\mathcal{P}}

\newcommand{\cS}{\mathcal{S}}

\newcommand{\bA}{\mathbb{A}}

\newcommand{\bG}{\mathbb{G}}
\newcommand{\bH}{\mathbb{H}}

\newcommand{\bK}{\mathbb{K}}

\newcommand{\bN}{\mathbb{N}}

\newcommand{\bQ}{\mathbb{Q}}
\newcommand{\bR}{\mathbb{R}}

\newcommand{\bZ}{\mathbb{Z}}

\newcommand{\gog}{\mathfrak{g}}

\newcommand{\gol}{\mathfrak{l}}

\newcommand{\ra}{\rightarrow}

\newcommand{\qand}{\quad \textrm{and} \quad}

\def\acts{\curvearrowright}
\newcommand\subsetsim{\mathrel{%
\ooalign{\raise0.2ex\hbox{$\subset$}\cr\hidewidth\raise-0.8ex\hbox{\scalebox{0.9}{$\sim$}}\hidewidth\cr}}}
\newcommand{\eps}{\varepsilon}

\DeclareMathOperator{\dist}{dist}

\DeclareMathOperator{\Ad}{Ad}

\DeclareMathOperator{\Lip}{Lip}
\DeclareMathOperator{\supp}{supp}

\DeclareMathOperator{\id}{id}
\DeclareMathOperator{\Vol}{Vol}

\theoremstyle{theorem}
\newtheorem{theorem}{Theorem}[section]
\newtheorem{corollary}[theorem]{Corollary}
\newtheorem{proposition}[theorem]{Proposition}
\newtheorem{lemma}[theorem]{Lemma}

\theoremstyle{definition}

\newtheorem{remark}[theorem]{Remark}

\begin{document}

\title{Quantitative multiple mixing}

\author{Michael Bj\"orklund}
\address{Department of Mathematics, Chalmers, Gothenburg, Sweden}
\email{micbjo@chalmers.se}

\author{Manfred Einsiedler}
\address{Department of Mathematics, ETH, Z\"urich, Switzerland}
\email{manfred.einsiedler@math.ethz.ch}

\author{Alexander Gorodnik}
\address{School of Mathematics, University of Bristol, Bristol, UK}
\email{a.gorodnik@bristol.ac.uk}



\date{}


\maketitle

\begin{abstract}
We develop a method for providing quantitative estimates for higher order correlations
 of group actions. In particular, we establish effective mixing of all orders for 
 actions of semisimple Lie groups as well as semisimple $S$-algebraic groups and semisimple adele groups.
 As an application, we deduce existence of approximate configurations in lattices
 of semisimple groups.
\end{abstract}

\section{Introduction}

The aim of this paper is to investigate behaviour of higher order correlations
for group actions.
Let us consider a measure-preserving action of a locally compact group $G$ on a probability measure space $(X,m)$. Given a test-function $\phi\in \cL^\infty(X)$,
we obtain a family of functions on $X$
$$
g\cdot \phi: x\mapsto \phi(g^{-1}\cdot x),\quad g\in G,
$$
generated by the group action. 
We may think about $\{g\cdot \phi:\, g\in G\}$ as a collection of random variables
on $(X,m)$. For chaotic group actions, it is natural to expect that these
random variables are asymptotically independent.
The independence property is measured by the correlations of the form
$$
m((g_1 \cdot \phi) \cdots (g_k \cdot \phi))=\int_X \phi(g_1^{-1}\cdot x)\cdots \phi(g_k^{-1}\cdot x)\, dm(x),
$$
where $g_1,\ldots, g_k\in G$.
We say that the action is \emph{mixing of order $k$} if for every 
$\phi_1,\ldots, \phi_k\in \cL^\infty(X)$,
$$
m((g_1 \cdot \phi_1) \cdots (g_k \cdot \phi_k))\longrightarrow 
m(\phi_1) \cdots m(\phi_k)
$$
as $g_i^{-1}g_j\to \infty$ in $G$ for every $i\ne j$.
It is a difficult problem in general to establish mixing of higher order.
It is not known, for instance, whether for $\bZ$-actions mixing of order two
implies mixing of order three, and there are examples of $\bZ^2$-actions which are mixing of order two, but not mixing of order three (see \cite{led}).
In this paper we develop a method which allows to obtain quantitative estimates on correlations of order $k$ inductively assuming only information about correlations of order two. While our interest is mostly in actions of semisimple Lie groups
and semisimple algebraic groups, it will apparent from the proof that the developed
method can be 
potentially applied more generally provided that there is a collection of one-parameter subgroups which satisfies certain regularity, growth, and mixing assumptions.

The multiple mixing property has been extensively studied for 
flows on homogeneous spaces of the form 
$X=\Gamma \backslash L$, where $L$ is a connected Lie group,
and $\Gamma$ is a lattice subgroup of $L$. 
We consider the left action of $L$ on $X$ defined by 
\begin{equation}
\label{eq:action}
l\cdot x=xl^{-1}\quad\hbox{ for $l\in L$ and $x\in X$.}
\end{equation}
It follows from the work of Dani \cite{dani1,dani2}
that under mild assumptions, any partially hyperbolic one-parameter flow on the space $X$
satisfies the Kolmogorov property, so that it is mixing of all orders.
It was conjectured by Sinai in \cite{sinai} that the horocycle flow is also mixing of all orders. Although mixing of order two for the horocycle flow is relatively easy to prove using representation-theoretic techniques (see \cite{para}), Sinai's conjecture was proved in full generality  only much later by Marcus in \cite{marcus}.
A strikingly general result about mutiple mixing was established by
Mozes in \cite{mozes}.
He shows that for arbitrary measure-preserving actions of a connected Lie group $G$,
mixing of order two implies mixing of all orders provided that the group $G$ is $\hbox{Ad}$-proper (namely, it has finite centre, and its image under the adjoint
map into the group $\hbox{Aut}(\hbox{Lie}(G))$ is closed). This, in particular,
applies to connected semisimple Lie group with finite centre.
Using Ratner's measure classification, Starkov in \cite{starkov}
proved mixing of all orders for general mixing one-parameter flows on
finite-volume homogeneous spaces.

Although quantitative estimates for the correlations of order two have substantial history,  it seems that there has been very little known regarding correlations of higher order. 
We intend to remedy this gap in the present paper.
We note that analysis of higher order correlations 
arises naturally in many combinatorial, arithmetic, and probabilistic problems.
In Section \ref{sec:conf}, we use our results to deduce existence of approximate configurations in lattice subgroups. We also apply our results 
in the forthcoming works \cite{BEG} and \cite{BG}
to establish quantitative estimates on the number rational points lying on compactifications of certain homogeneous algebraic varieties,
and to derive limit theorems describing fine statistical properties
of group actions.

\subsection{Semisimple Lie groups}\label{s:ss}

Let $G$ be a connected semisimple Lie group with finite centre. We observe that given a measure-preserving action of $G$
on a probability space $(X,m)$, the correlations of order two can be interpreted as matrix
coefficients of the corresponding unitary representation of $G$ on $\cL^2(X)$.
Starting with the research programme of Harish-Chandra (summarised in the monographs \cite{war1,war2}), properties of 
matrix coefficients for representations of semisimple Lie groups have
been extensively studied. In particular, we mention important works 
of Borel and Wallach \cite{BW}, Cowling \cite{cowling}, Howe \cite{h}, 
Li and Zhu \cite{li,li2}, Moore \cite{m}, and  Oh \cite{oh1,oh2}
that provide explicit estimates on matrix coefficients for semisimple groups.
We formulate the main estimate coming for these works that will be a starting point of our investigation. 
We fix a left-invariant Riemannian metric $\rho_G$ on $G$ which is bi-invariant 
under a fixed maximal compact subgroup $K$ of $G$.
Let $\pi:G\to \mathcal{U}(H)$ be a unitary representation of $G$.
We say that $\pi$ has \emph{strong spectral gap}
if the restriction of $\pi$ to every noncompact simple factor of $G$ is isolated from the trivial representation with respect to the Fell topology on the dual space.
For every representation $\pi$ with the strong spectral gap, 
there 
exist $C,\delta>0$ such that 
for every $K$-finite vectors $v_1,v_2\in H$,
\begin{equation}\label{eq:exp_mix0}
\left<\pi(g)v_1,v_2\right>\le C\, e^{-\delta\, \rho_G(g,e)}\,\, \cN(v_1) \cN(v_2),
\end{equation}
where $\cN(v)=(\dim \left<Kv\right>)^{1/2} \|v\|$.
It is important for applications to have an analogue of
the estimate \eqref{eq:exp_mix0} which is valid for all smooth vectors in $H$.
It was observed by Katok and Spatzier in \cite{KS}
that under the strong spectral gap assumption, 
there exists $\delta>0$ such that
for all sufficiently large integers $d$
and arbitrary smooth vectors $v_1,v_2\in H$,
\begin{equation}\label{eq:exp_mix0_sobolev}
\left<\pi(g)v_1,v_2\right>\ll_d\, e^{-\delta\, \rho_G(g,e)}\,\, \|\cC_G^d v_1\| \|\cC_G^d v_2\|,
\end{equation}
where $\cC_G$ denotes the Casimir differential operator for $G$.

Now we suppose that the group $G$ is a closed subgroup 
of a connected Lie group $L$. Let $\Gamma$ be a lattice subgroup in $L$
and $X=\Gamma\backslash L$ equipped with the invariant probability measure $m$.
We consider the left action of $G$ on $X$
defined by \eqref{eq:action}.
We say that this action has \emph{strong spectral gap}
if the corresponding unitary representation of $G$ on $\cL^2_0(X)$
has strong spectral gap.
If this is the case, then the estimate \eqref{eq:exp_mix0_sobolev}
implies, in particular, that there exists $\delta>0$ such that
for all sufficiently large $d$, functions $\phi_1,\phi_2\in\cC_c^\infty(X)$, and an element $g\in G$,
\begin{equation}\label{eq:exp_mix}
| m((g \cdot \phi_1) \phi_2) - m(\phi_1) m(\phi_2) |
\ll_d
e^{-\delta \, \rho_G(g,e)} \,\, 
\|\cC_G^d \phi_1\|_2\, \|\cC_G^d \phi_2\|_2.
\end{equation}

Our first main result gives quantitative estimate on correlations
of arbitrary order for semisimple Lie groups generalising \eqref{eq:exp_mix}.
We formulate this estimate in terms of the Sobolev norms introduced in Section \ref{sec:sobolev} below.

\begin{theorem}[Exponential mixing of all orders for Lie groups]
\label{main1}
Let $L$ be a connected Lie group, $\Gamma$ a lattice subgroup of $L$,
and $X=\Gamma\backslash L$ equipped with the invariant probability measure $m$.
Let $G$ be a connected semisimple Lie subgroup of $L$ with finite center. We assume that the action of $G$ on $X$ has strong spectral gap.

Then, for every $k \geq 2$ and sufficiently large $d$, there exists $\delta=\delta(k,d) > 0$ such that for all functions $\phi_1,\ldots,\phi_k \in\cC_c^\infty(X)$ and elements $g_1,\ldots,g_k \in G$,
we have
\[
| m((g_1 \cdot \phi_1) \cdots (g_k \cdot \phi_k)) - m(\phi_1) \cdots m(\phi_k) |
\ll_{d,k}
\mathfrak{M}(g_1,\ldots,g_k)^{-\delta} \,\, \cS_d(\phi_1) \cdots \cS_d(\phi_k),
\]
where
$$
\mathfrak{M}(g_1,\ldots,g_k):=\exp\left(\min_{i\ne j}\rho_G(g_i,g_j)\right).
$$
\end{theorem}

Our result should be compared with the recent work of Konstantoulas \cite{konst}
which also provides an estimate of the form
\begin{equation}
\label{eq:mix_konst}
| m((a_1 \cdot \phi_1) \cdots (a_k \cdot \phi_k)) - m(\phi_1) \cdots m(\phi_k) |
\le 
\mathfrak{R}(a_1,\ldots,a_k) \,\, C(\phi_1,\ldots,\phi_k)
\end{equation}
with explicit $\mathfrak{R}(a_1,\ldots,a_k)$,
where the elements $a_1,\ldots,a_k$ belong to the same Cartan subgroup of $G$.
This estimate in \cite{konst} holds on a dense subspace of functions,
but it seems that the method of the proof in \cite{konst} cannot be used to make this subspace explicit. In particular, the constant $C(\phi_1,\ldots,\phi_k)$ in \eqref{eq:mix_konst} is not explicit.
The estimator $\mathfrak{R}(a_1,\ldots,a_k)$ is different from our estimator
$\mathfrak{M}(a_1,\ldots,a_k)^{-\delta}$. In particular, it might happen that 
$\mathfrak{R}(a_1,\ldots,a_k)\nrightarrow 0$ when $a_i^{-1}a_j\to \infty$ for all $i\ne j$, so that the estimate \eqref{eq:mix_konst} does not imply mixing 
of order $k$ along the Cartan subgroup.
On the other hand, probably it might happen that 
$\mathfrak{R}(a_1,\ldots,a_k)\le \mathfrak{M}(a_1,\ldots,a_k)^{-\delta}$ for some particular choices
of elements $a_1,\ldots,a_k$.
We note that validity of our estimate 
in Theorem \ref{main1} for general elements $g_1,\ldots,g_k\in G$ is crucial
for the combinatorial application discussed in Section \ref{sec:conf0} below.

We also mention that the exponential multiple mixing estimates
have been established for 
partially hyperbolic flows (see \cite[Th.~4.4]{hui} and \cite[Th.~2]{dolg}),
but it is not clear how to extend this approach to more general group actions. 
\\

We note that the strong spectral gap assumption in Theorem \ref{main1} is known to hold in a number of cases.
For instance, if a simple factor $G_0$ of $G$ has rank at least two,
then if the action of $G_0$ on $X$ is ergodic, it follows from the Kazhdan property (T) that the representation of $G_0$ on $\cL^2_0(X)$ is isolated from
the trivial representation. Another important case is when $L$ is a connected
semisimple Lie group with finite centre and no compact factors,
and $\Gamma$ is an irreducible lattice in $L$.
Then the action of $L$ on 
$X=\Gamma\backslash L$ has strong spectral gap (see \cite[p.~285]{KSar}).
Furthermore, 
for the homogeneous spaces of this form,
given any closed 
connected semisimple subgroup $G$ of $L$, the action of $G$ on 
$X$ also has strong spectral gap (see \cite{n}).\\

We observe that the correlations of order $k$ can be interpreted 
it terms of the probability measure $m_{\Delta_k(X)}$ 
supported on the diagonal $\Delta_k(X)$ in $X^k$ which is the push-forward of $m$
under the diagonal map $X\to \Delta_k(X)\subset X^k$. We note that the measure $m_{\Delta_k(X)}$ is invariant under the action
of the diagonal subgroup $\Delta_k(G)$ of $G^k$,
and its projections to each of the factors of $X^k$ are equal to $m$.
More generally, we say that a probability measure $\xi$ on $X^k$
is a \emph{$k$-coupling} of $(X,m)$ if its marginals (the push-forwards of $\xi$ onto the factors of $X^k$) are equal to $m$. We establish the following effective equiditribution result that applies to general $\Delta_k(G)$-invariant $k$-couplings.

\begin{theorem}[Uniform exponential mixing of all orders for Lie groups]
	\label{main1_coupl}
	Let $G, X, m$ be as in Theorem \ref{main1}. Then, for every $k \geq 2$ and sufficiently  large $d$,
	there exists $\delta=\delta(k,d)>0$ such that 
	for every $\Delta_{k}(G)$-invariant coupling $\xi$ of $(X,m)$,
	functions $\phi_1,\ldots,\phi_k\in\cC_c^\infty(X)$, and
	elements $g_1,\ldots,g_k \in G$, we have
	\[
	|\xi((g_1\cdot \phi_1)\otimes\cdots \otimes(g_k\cdot \phi_k))-m(\phi_1)\cdots m(\phi_k)|
\ll_{d,k}  \,
	 \mathfrak{M}(g_{[k]})^{-\delta}\, \cS_d(\phi_1)\cdots \cS_d(\phi_k).
	\]
	In particular, the above bound is uniform over all $\Delta_{k}(G)$-invariant $k$-couplings $\xi$ of $(X,m)$.
\end{theorem}

The proofs of Theorem \ref{main1} and Theorem \ref{main1_coupl} will be given in Section \ref{outline_main1_0}.

\subsection{An application: approximate configurations in lattices}
\label{sec:conf0}

It was realised by Furstenberg in his proof of Szemeredi theorem \cite{F}
that analysis of higher order correlations of dynamical systems leads to deep combinatorial
consequences. Developments of these ideas have allowed to prove a number of far-reaching results regarding existence of configurations. For instance, we mention the works of Furstenberg, Katznelson and Weiss \cite{FKW}
and Ziegler \cite{Z} which show that  given a subset $\Omega$ of $\bR^n$ of positive upper density and a $k$-tuple $(v_1,\ldots,v_k)\in (\bR^n)^k$, for all sufficiently large dilations $t$
and $\eps>0$ there exist a $k$-tuple $(\omega_1,\ldots,\omega_k)\in\Omega^k$ and  an isometry $I$ of $\bR^n$ such that
$$
d(tv_i, I(\omega_i))<\eps\quad\quad\hbox{for all $i=1,\ldots,k$,}
$$
i.e., the set $\Omega$ must contain an approximate isometric copy of any
sufficiently dilated configuration. 
In particular, it follows from this result that given any lattice $\Lambda$ in $\bR^n$,
any $k$-tuple $(v_1,\ldots,v_k)\in (\bR^n)^k$ and $\eps>0$,
for all sufficiently large $t$, 
there exist a $k$-tuple $(z_1,\ldots,z_k)\in\Lambda^k$ and  an isometry $I$ of $\bR^n$ such that
$$
d(tv_i, I(z_i))<\eps\quad\quad\hbox{for all $i=1,\ldots,k$.}
$$
It would be interesting to investigate whether an analogue of this statement holds
for other locally compact groups and whether it can be made explicit in terms of $t$.
Here we address these questions for lattices in semisimple Lie groups.\\

To illustrate our general result, let us consider a Fuchsian group $\Gamma\subset \hbox{Isom}(\bH^2)$ of finite covolume. 
For fixed $v_0\in \bH^2$, we consider a discrete subset $\Gamma v_0$
of the hyperbolic plane $\bH^2$. How rich is the set of $k$-tuple
$(z_1,\ldots,z_k)$ with $z_i\in \Gamma v_0$ ?
For a $k$-tuple $(v_1,\ldots,v_k)\in \bH^2$,
we define its width by
$$
w(v_1,\ldots,v_k)=\min_{i\ne j} d(v_i,v_j).
$$
It follows from our main result that for every $k\ge 2$,
given an arbitrary $k$-tuple $(v_1,\ldots,v_k)\in (\bH^2)^k$
that satisfies 
\begin{equation}
\label{eq:min}
w(v_1,\ldots,v_k)\ge c_k\log(1/\eps)
\end{equation}
with $\eps\in (0,\eps_k)$,
there exist a $k$-tuple $(z_1,\ldots,z_k)\in (\Gamma v_0)^k$ and an isometry $g\in \hbox{PSL}_2(\bR)$
such that 
$$
d(v_i,g\,z_i)<\eps\quad\quad \hbox{for all $i=1,\ldots,k$.}
$$

We note that the instance of this result when $k=2$ reduces to analysing the set of distances
$\{d(\gamma\, v_0,v_0):\, \gamma\in \Gamma\}$. For example, when $\Gamma=\hbox{PSL}_2(\bZ)$,
we need to show that for 
$$
\mathcal{D}:=\{d(\gamma\, v_0,v_0):\, \gamma\in \Gamma\}
=\left\{\cosh^{-1} 2(a^2+b^2+c^2+d^2)/4:\,\, ad-bc=1,\, a,b,c,d\in\bZ\right\},
$$
the sets $\mathcal{D}\cap [c_k\log(1/\eps)),\infty)$
are $\eps$-dense in $[c_k\log(1/\eps),\infty)$.
Using that the set of distances is contained in $\cosh^{-1}(\bN)/4$, it is not hard to check that
if \eqref{eq:min} is replaced by the condition that $w(v_1,\ldots,v_k)\ge \sigma(\eps)$ with 
$\sigma(\eps)=o(\log(1/\eps))$
as $\eps\to 0^+$, then the above statement fails.\\

In full generality, we consider a connected semisimple Lie group $G$ with finite centre and without compact factors equipped with 
a left-invariant Riemannian metric $\rho_G$ on $G$ which is
bi-invariant under a fixed maximal compact subgroup.
For any irreducible lattice $\Gamma$ in $G$, we prove

\begin{corollary}
\label{main1_cor}
For every $k\ge 2$, there exist $c_k, \eps_k>0$ such that
given arbitrary $k$-tuple $(g_1,\ldots, g_k)\in G^k$
satisfying 
$$
w(g_1,\ldots,g_k):=\min_{i\ne j} \rho_G(g_i,g_j)\ge c_k\log(1/\eps)
$$
with $\eps\in (0,\eps_k)$, there exist a $k$-tuple $(\gamma_1,\ldots, \gamma_k)\in \Gamma^k$
and $g\in G$ such that 
$$
\rho_G(g_i,g\,\gamma_i)<\eps\quad\quad \hbox{for all $i=1,\ldots,k$.}
$$
\end{corollary}

We prove Corollary \ref{main1_cor} in Section \ref{sec:conf}.

\subsection{S-algebraic groups}

The results of Section \ref{s:ss} can be extended to actions
of $S$-algebraic semisimple groups. Let $\bG\subset \hbox{GL}_n$ be a simply connected absolutely simple algebraic group defined over a number field $F$. We denote by $\mathcal{V}_F$ the set of places of $F$, 
and for $v\in \mathcal{V}_F$ we write $F_v$ for the completion of $F$ with respect to the norm $|\cdot|_v$. Let $G_v=\bG(F_v)$. We fix a finite subset $S$ of $\mathcal{V}_F$ and consider the group 
\begin{equation}
\label{eq:GS}
G:=\prod_{v\in S} G_v.
\end{equation}
Let $S=S_\infty\sqcup S_f$ where $S_\infty$ and $S_f$ denote the subsets of the Archemedean places and the non-Archemedean places respectively.
We set 
$$
G_\infty:=\prod_{v\in S_\infty} G_v\quad\hbox{and}\quad G_f:=\prod_{v\in S_f} G_v,
$$
so that $G=G_\infty G_f$.

Let us consider a measure-preserving action of $G$ on a probability space $(X,m)$.
Then we obtain a unitary representation of $G$ on the space $\cL^2(X)$.
Given a compact open subgroup $U$ of $G_f$, we denote by $\cC^\infty(X)^U$
the subalgebra of $\cL^2(X)$ consisting of vectors which are smooth with respect to
the action of $G_\infty$ and are $U$-invariant.
We say that the action of $G$ on $(X,m)$
has \emph{strong spectral gap} if the representation of each noncompact factor $G_v$ with $v\in S$
on $\cL^2_{0}(X)$ is isolated  from the trivial representation.
In this situation there are quantitative bounds on matrix coefficients of $\cC^\infty(X)^U$
analogous to \eqref{eq:exp_mix0}. In particular, we refer to the works 
of Borel, Wallach \cite{BW}, Oh \cite{oh2}, Clozel, Oh, Ullmo \cite{CUO},
and Gorodnik, Maucourant, Oh \cite{GMO}
that discuss such bounds over non-Archimedean fields.
For every $v\in S$, let us fix a norm on $\hbox{M}_n(F_v)$ and define the height function on $G$ by
$$
\hbox{H}(g):=\prod_{v\in S} \|g_v\|_v\quad \hbox{ for $g=(g_v)_{v\in S}\in G$.} 
$$
One can check that $\hbox{H}$ is a proper function on $G$.
With this notation, there exists $\delta>0$
such that for all sufficiently large $d$, a compact open subgroup $U$ of $G_f$, 
functions $\phi_1,\phi_2\in \cC^\infty(X)^U$, and an element $g\in G$,
\begin{equation}
\label{eq:mixH}
|m((g\cdot \phi_1)\phi_2) -m(\phi_1)m(\phi_2)|\ll_{d,U} \hbox{H}(g)^{-\delta}\, \|\mathcal{C}_{G_\infty}^d\phi_1\|_2\, \|\mathcal{C}_{G_\infty}^d\phi_2\|_2.
\end{equation}
This estimate can deduces as in the proof of \cite[Theorem~3.27]{GMO} from the bounds 
for representations of the local factors $G_v$. In this paper, we establish an analogous estimate for correlations
of higher order.

We consider a continuous measure-preserving action of $G$ on a locally compact Hausdorff space $X$
equipped with a probability Borel measure $m$. Let $\cC_c^\infty(X)$ 
denote the subalgebra of $\cC_c(X)$ consisting of functions which are smooth with
respect to the action of $G_\infty$ and invariant under a compact open subgroup of $G_f$.
Given a compact open subgroup $U$ of $G_f$, we denote by $\cC_c^\infty(X)^U$
the subalgebra of functions in $\cC_c^\infty(X)$ which are invariant under $U$.

We assume that there is a family of norms $\cS_d$, $d\in \mathbb{N}$, on $\cC_c^\infty(X)^U$
that satisfy the following properties:
\begin{enumerate}
	\setlength\itemsep{0.5em}
	\item[N1.] For all sufficiently large $d$, any compact open subgroup $U$ of $G_f$, and $\phi \in\cC_c^\infty(X)^U$,
	\begin{equation}
	\label{a1}
	\|\phi\|_\infty \ll_{d,U}  \cS_d(\phi).
	\end{equation}
	
	\item[N2.] For all sufficiently large $d$, any compact open subgroup $U$ of $G_f$, $\phi \in\cC_c^\infty(X)^U$, and $g\in G_v$ with $v\in S_\infty$,
	\begin{equation}
	\label{a2}
	\|\phi - g \cdot \phi\|_\infty \ll_{d,U} \rho_{G_v}(g,e_{G_v}) \, \cS_d(\phi),
	\end{equation}
	where $\rho_{G_v}$ denotes a left-invariant Riemannian metric on $G_v$.
	
	\item[N$2^\prime$.] For all sufficiently large $d$, any compact open subgroup $U$ of $G_f$, $\phi \in\cC_c^\infty(X)^U$, and $g\in G_v$ with $v\in S_f$,
	\begin{equation}
	\label{a2'}
	\|\phi - g \cdot \phi\|_\infty \ll_{d,U} \|\Ad(g)-id\| \, \cS_d(\phi),
	\end{equation}
	where $\|\cdot \|$ denotes the operator norm on $\hbox{End}(\hbox{Lie}(G_v))$
	for a fixed choice of a norm on $\hbox{Lie}(G_v)$.
	
	\item[N3.] For all sufficiently large $d$, there exists $\sigma=\sigma(d)>0$
	such that for any compact open subgroup $U$ of $G_f$, $\phi \in\cC_c^\infty(X)^U$, and $g\in G$,
	\begin{equation}
	\label{a3}
	\cS_d(g \cdot \phi) \ll_{d,U} \|\Ad(g)\|^{\sigma} \, \cS_d(\phi).
	\end{equation}
	
	\item[N4.] There exists $r > 0$ such that for all sufficiently large $d$, any compact open subgroup $U$ of $G_f$,
	and $\phi_1, \phi_2 \in\cC_c^\infty(X)^U$,
	\begin{equation}
	\label{a4}
	\cS_d(\phi_1 \phi_2) \ll_{d,U} \cS_{d+r}(\phi_1) \, \cS_{d+r}(\phi_2).
	\end{equation}
\end{enumerate}

Such collections of norms can constructed on finite-volume homogeneous spaces
of $S$-algebraic groups (see, for instance, \cite[Appendix~A]{EMMV}).

We establish the following general result which extends the estimate \eqref{eq:mixH} to correlations of arbitrary order.

\begin{theorem}[Exponential mixing of all orders for $S$-algebraic groups]
	\label{main1_Sarithmetic}
	
Let $G$ be an $S$-algebraic group as in \eqref{eq:GS}
which acts continuously and in a measure-preserving fashion on
a locally compact Hausdorff space $X$ equipped with a Borel probability measure $m$. We assume 
that $X$ is equipped with a family of norms $\cS_d$ satisfying Properties {\rm N1--N4},
and  
there exists $\delta_2>0$
such that for all sufficiently large $d$, a compact open subgroup $U$ of $G_f$, 
functions $\phi_1,\phi_2\in\cC_c^\infty(X)^U$, and an element $g\in G$, we have 
$$
|m((g\cdot \phi_1)\phi_2) -m(\phi_1)m(\phi_2)|\ll_{d,U} \hbox{\rm H}(g)^{-\delta_2}\, \cS_d(\phi_1)\cS_d(\phi_2).
$$
	
Then, for every $k \geq 2$ and sufficiently large $d$, there exists $\delta=\delta(k,d, \delta_2) > 0$
such that for all compact open subgroups $U$ of $G_f$,
functions $\phi_1,\ldots,\phi_k \in\cC_c^\infty(X)^U$, and elements $g_1,\ldots,g_k \in G$, we have 
	\[
	| m((g_1 \cdot \phi_1) \cdots (g_k \cdot \phi_k)) - m(\phi_1) \cdots m(\phi_k) |
	\ll_{d,U,k}
	\mathfrak{H}(g_1,\ldots,g_k)^{-\delta} \, \cS_d(\phi_1) \cdots \cS_d(\phi_k),
	\]
	where 
	$$
	\mathfrak{H}(g_1,\ldots,g_k):=\min_{i\ne j} \hbox{\rm H}(g_i^{-1} g_j).
	$$
\end{theorem}

In fact, our method allows to deal with arbitrary 
$\Delta_{k}(G)$-invariant $k$-couplings of the space $(X,m)$.

\begin{theorem}[Uniform exponential mixing of all orders for $S$-algebraic groups]
	\label{main1_coupl_Sarithmetic}
	Let $G, X, m$ be as in Theorem \ref{main1_Sarithmetic}. Then, for every $k \geq 2$
	and sufficiently large $d$, there exists $\delta=\delta(k,d,\delta_2)>0$ such that for all compact open subgroups $U$ of $G_f$, every $\Delta_{k}(G)$-invariant coupling $\xi$ of $(X,m)$,
	functions $\phi_1,\ldots,\phi_k \in\cC_c^\infty(X)^U$,
	and elements $g_1,\ldots,g_k \in G$, we have
		\[
		| \xi((g_1 \cdot \phi_1)\otimes \cdots \otimes(g_k \cdot \phi_k)) - m(\phi_1) \cdots m(\phi_k) |
		\ll_{d,U,k}
		\mathfrak{H}(g_1,\ldots,g_k)^{-\delta} \, \cS_d(\phi_1) \cdots \cS_d(\phi_k).
		\]
	In particular, the above bound is uniform over all $\Delta_{k}(G)$-invariant $k$-couplings $\xi$ of $(X,m)$.
\end{theorem}

The proofs of Theorem \ref{main1_Sarithmetic} and \ref{main1_coupl_Sarithmetic} will be given in Section \ref{sec:salg}. We note that the uniformity in Theorem \ref{main1_coupl_Sarithmetic}
will be crucial in our analysis of higher order correlations on adele groups 
in the next section.

\subsection{Adele groups}

Let $\bG\subset \hbox{GL}_n$ be a simply connected absolute simple algebraic group defined over a number field $F$. Let $\bG(\bA_F)$ be the corresponding adele group
and 
$$
X:=\bG(F)\backslash\bG(\bA_F)
$$
equipped with the invariant probability measure $m$.
For each $v\in \mathcal{V}_F$, we fix a norm $\|\cdot \|_v$ on $\hbox{M}_n(F_v)$
which is the $\max$ norm for almost all places $v$. 
The height function $\hbox{H}: \bG(\bA_F)\to \mathbb{R}^+$ is defined by
\begin{equation}
\label{eq:HHH}
\hbox{H}(g):=\prod_{v\in\mathcal{V}_F} \|g\|_v, \quad \hbox{for $g=(g_v)_{v\in\mathcal{V}_F}\in \bG(\bA_F)$}.
\end{equation}
We note that $\hbox{H}$ is a proper function on $\bG(\bA_F)$ (see, for instance, \cite[Lemma~2.5]{GMO}).

We denote by $G_\infty$ the product of $\bG(F_v)$ over the Archemdean places $v$ and 
by $G_f$ the group of finite adeles. Also 
we denote by $U_\infty$ 
the product of $\bG(F_v)$ over the Archemedean places $v$ such that 
$\bG(F_v)$ is compact.
Given a subgroup $U$ of $\bG(\bA_F)$, we denote by $\cC_c^\infty(X)^U$
the algebra of compactly supported functions on $X$ which are smooth with respect to the action of $G_\infty$ and are $U$-invariant. When $W$ is a compact open subgroup of $G_f$,
we introduce a family of Sobolev norms $\cS_{d,W}$ on the algebras $\cC_c^\infty(X)^W$
(see Section \ref{outline_main1_adele}).
We establish the following generalisation of \cite[Theorem~3.27]{GMO}
for $U_\infty$-invariant functions.

\begin{theorem}[Exponential mixing of all orders for adele groups]
\label{main2}
Let $\bG$ be a simply connected absolutely simple linear algebraic group defined over 
a number field $F$  and $X = \bG(F) \backslash \bG(\bA_F)$
equipped with the invariant probability measure $m$ on $X$.
We assume that $\bG$ is isotopic over $F_v$ for some Archemedian $v\in\mathcal{V}_F$.

Then, for every $k \geq 2$ and sufficiently large $d$, there exists $\delta=\delta(k,d) > 0$ such
that for every compact open subgroup $W$ of $G_f$, $U_\infty$-invariant functions $\phi_1, \ldots, \phi_k \in \cC^\infty_c(X)^{W}$, and elements $s_1,\ldots,s_k \in \bG(\bA_F)$,
we have
\[
| m((s_1 \cdot \phi_1) \cdots (s_k \cdot \phi_k)) - m(\phi_1) \cdots m(\phi_k) |
\ll_{d,W,k}
\mathfrak{H}(s_1,\ldots,s_k)^{-\delta} \, \cS_{d,W}(\phi_1) \cdots \cS_{d,W}(\phi_k),
\]
where 
$$
\mathfrak{H}(s_1,\ldots,s_k):=\min_{i\ne j} \hbox{\rm H}(s_i^{-1} s_j).
$$
\end{theorem}

Since $\hbox{H}$ is a proper function on $\bG(\bA_F)$, Theorem \ref{main2} in particular implies that the action of $\bG(\bA_F)$ on $X = \bG(F) \backslash \bG(\bA_F)$
is mixing of all orders. This was previously established in \cite{GBTT},
but the method in \cite{GBTT} relies on the theory of unipotent flows
and does not provide any explicit estimates.
In \cite{BEG}, we apply Theorem \ref{main2} to establish effective
counting estimate for the number of rational points
lying on the compactifications of the varieties of the form $\bG^k/\Delta_k(\bG)$.

It is quite likely that the assumption in Theorem \ref{main2} that $\bG$ is isotopic over $F_v$ for some Archemedian $v\in\mathcal{V}_F$ can be removed. It is needed because 
our argument relies on the results from \cite{EMV} 
which are only proved for real homogeneous space. 
Once an $S$-algebraic version of \cite{EMV} is developed,
Theorem \ref{main2} will follow for general $\bG$ using our method.

The proof of Theorem \ref{main2} will be given in Section \ref{outline_main1_adele}.

\subsection{Organisation of the paper}
In Section \ref{outline_main1_0}, we discuss higher order correlations for 
semisimple Lie groups. 
In particular, we reformulate our main results in terms of the Wasserstein distance
for couplings and prove the results from \S\ref{s:ss}. 
Next, we apply the established correlation
estimates in Section \ref{sec:conf} to deduce Corollary \ref{main1_cor} regarding existence of approximate configurations in lattice subgroups. 
In Section \ref{sec:salg} we analyse higher order correlations for $S$-algebraic groups,
and in Section \ref{outline_main1_adele} --- for adele groups.
The proofs in Sections \ref{outline_main1_0} and \ref{sec:salg}
rely on a general inductive estimate for couplings 
(Proposition \ref{mainest}) which is established in Section \ref{PrfABCDT}.
It will become apparent in Section \ref{PrfABCDT} that our method
can be applied more generally to study couplings which are invariant under a flow
satisfying suitable regularity, growth, and mixing properties.
Also in Section \ref{sec:wasserstein}, we discuss basic properties of 
the Wasserstein distance which are used in the paper.

\section{Higher-order correlations for semisimple Lie groups}
\label{outline_main1_0}


\subsection{Preliminaries}
\label{Subsec:notation}
Let $G$ be a connected semisimple Lie group with finite center.
We fix a Cartan subgroup $A$ of $G$. We denote by $\Sigma\subset \hbox{Hom}(A,\mathbb{R}_+^\times)$  the root system with respect to the adjoint action of $A$ on the Lie algebra $\gog=\hbox{Lie}(G)$.
Then we have the root space decomposition 
\begin{equation}
\label{eq:roots}
\gog = \gog^0 + \bigoplus_{\alpha \in \Sigma} \gog^\alpha,
\end{equation}
where $\gog^0$ is the centraliser of  the Lie algebra of $A$ in $\gog$, and 
\[
\gog^\alpha := \big\{ Z \in \gog \, : \,\, \Ad(a)Z = \alpha(a) Z, \enskip \textrm{for all $a \in A$}  \big\}.
\]
We fix a choice of the subset $\Sigma^+\subset \Sigma$ of positive roots and denote by 
$$
A^+:=\big\{a\in A:\, \alpha(a)\ge 1, \enskip \textrm{for all $\alpha \in \Sigma^+$}  \big\}
$$
the corresponding closed positive Weyl chamber in $A$. There exists a maximal compact subgroup $K$ of 
$G$ such that the \emph{Cartan decomposition} 
\begin{equation}
\label{eq:cartan}
G=KA^+K
\end{equation}
holds. 
It is a standard fact (see e.g. \cite[Ch.~9]{Hel}) that 
if $g=k_1 a_g k_2$ for $k_1, k_2 \in K$ and $a_g\in A^{+}$, then the component $a_g$ is unique.
We call the component $a_g$ the \emph{Cartan projection} of $g$. \\

Let $\langle \cdot, \cdot \rangle$ be an $\Ad(K)$-invariant inner product on $\gog$, and denote by $\| \cdot \|$ the 
corresponding norm on $\gog$. Let $\rho_G$ denote the left-invariant distance function on $G$ induced by the 
Riemannian metric corresponding to $\langle \cdot, \cdot \rangle$. 
We note that $\rho_G$ is bi-$K$-invariant.

We also define a sub-multiplicative function 
$\|\cdot\|_{\rm op}$ on $G$ by
\begin{equation}
\label{def*}
\|g\|_{\rm op} := \max\big\{ \|\Ad(g)Z\| \, : \, Z\in\gog\quad\hbox{with $\|Z\| = 1$} \big\}.
\end{equation}
We note that since $G$ is semisimple, every transformation $\hbox{Ad}(g)$ satisfies $\det(\hbox{Ad}(g))=1$, so that 
it has at least one eigenvalue whose absolute value is greater or equal to one. 
This implies that
$$
\|g\|_{\rm op}\ge 1\quad\hbox{ for all $g\in G$.}
$$

The following lemma summarises basic properties of the functions $\rho_G$ and $\|\cdot \|_{\rm op}$
that will be used in the proof.

\begin{lemma}
\label{lemma1_outline_main1}
\begin{enumerate}
  \setlength\itemsep{0.5em}

\item[(i)] For every $g\in G$,
$$
\|g\|_{\rm op} =\max_{\alpha \in \Sigma^{+}} \alpha(a_g),  
$$
where $a_g \in A^+$ denotes the Cartan projection of the element $g$.

\item[(ii)] For every $g\in G$, there exists $Z \in \gog$ 
such that $\Ad(Z)$ is nilpotent, $\|Z\| = 1$, and
\[
\|g\|_{\rm op} = \|\Ad(g)Z\|.
\]

\item[(iii)] There exist constants $c_1\ge 1$ and $c_2 > 0$ such that 
\[
c_1^{-1} \log \|g\|_{\rm op} - c_2 \leq \rho_G(g,e_G) \leq c_1 \log \|g\|_{\rm op} + c_2
\] 
for all $g \in G$.

\item[(iv)] There is a constant $c_3 \geq 1$ such that
for all every $X \in \gog $ such that $\hbox{\rm Ad}(X)$ is nilpotent,
\[
c_3^{-1}\, \max(1,\|X\|) \leq \|\exp(X)\|_{\rm op} \leq c_3\, \max(1,\|X\|)^{\dim(G)}.
\]
\end{enumerate}
\end{lemma}

In the proof of (iv), we will use the following lemma.
Since later we will also need a version of this lemma over $p$-adic local fields,
we formulate it more generally.

\begin{lemma}\label{l:norm}
	Fix a norm on $\hbox{\rm M}_n(\mathbb{K})$, where $\mathbb{K}$ is a locally compact normed field.
	Then there exists $c_0>0$ such that for every nilpotent matrix $X\in \hbox{\rm M}_n(\mathbb{K})$,
	$$
	\|\exp(X)\|\ge c_0\, \|X\|.
	$$
\end{lemma}

\begin{proof}
We fix a norm on $\mathbb{K}^n$. 
Since all the norms on $\hbox{\rm M}_n(\mathbb{K})$ are equivalent,
without loss of generality, we may assume that the norm in the lemma satisfies
\begin{equation}
\label{eq:normA}
\|Av\|\le \|A\|\, \|v\|\quad 
\hbox{for every $A\in \hbox{\rm M}_n(\mathbb{K})$ and $v\in \mathbb{K}^n$.}
\end{equation}
For a nilpotent matrix $X$, we set
$$
c(X):=\max\{\|Xv\|: v\in \mathbb{K}^n \hbox{ such that $\|v\|=1$ and $X^2v=0$} \}.
$$
Given any nilpotent $X\in \hbox{\rm M}_n(\mathbb{K})$
and $v\in \mathbb{K}^n$ such that $\|v\|=1$ and $X^2v=0$,
we have
\begin{align*}
\|\exp(X)\|\ge \|\exp(X)v\|= \|v+ X v\|\ge \|Xv\|-1.
\end{align*}
Hence, 
$$
\|\exp(X)\|\ge c(X)-1.
$$
We claim that
\begin{equation}
\label{eq:cx}
\inf\{c(X): \hbox{ nilpotent $X\in \hbox{\rm M}_n(\mathbb{K})$ such that 
	$\|X\|=1$}\}>0.
\end{equation}
Suppose that \eqref{eq:cx} fails. Since the function $c$ is continuous,
the infimum is achieved, and there exists nilpotent 
$X\in \hbox{\rm M}_n(\mathbb{K})$ with $\|X\|=1$ such that $c(X)=0$.
Then it follows that for every $v\in\mathbb{K}^n$, if
$X^2v=0$, then $Xv=0$, i.e., $\hbox{ker}(X^2)=\hbox{ker}(X)$.
This also implies that for every $\ell\ge 1$, we have
$\hbox{ker}(X^{\ell+1})=\hbox{ker}(X^\ell)$.
Since $X$ is nilpotent, we conclude that $X=0$,
but $\|X\|=1$. This contradiction shows that \eqref{eq:cx} holds.
Hence, there exists $c'_0>0$ such that for every nilpotent $X$,
$$
\|\exp(X)\|\ge c'_0\,\|X\|-1.
$$
Then $\|\exp(X)\|\ge c'_0/2\,\|X\|$ when $\|X\|\ge 2/c'_0$.

On the other hand, it is also clear that 
$\|\exp(X)\|^{-1} \|X\|$ is uniformly bounded when
$\|X\|\le 2/c'_0$. Indeed, since $\exp(X)$ is unipotent,
it follows from \eqref{eq:normA} that $\|\exp(X)\|\ge 1$.

Combining these two bounds completes the proof.
\end{proof}

\begin{proof}[Proof of Lemma \ref{lemma1_outline_main1}]
Since the norm on $\gog$ is $\hbox{Ad}(K)$-invariant, it is clear that
$\|g\|_{\rm op}=\|a_g\|_{\rm op}$. 
We note that the root decomposition \eqref{eq:roots} is orthogonal.
Decomposing an element $Y\in \gog$ with respect to \eqref{eq:roots}, we obtain that for $a\in A^+$,
\begin{align}\label{eq:norm0}
\|\hbox{Ad}(a)Y\|^2&=\sum_{\alpha\in \Sigma\cup \{0\}} \alpha(a)^2\|Y_\alpha\|^2\le \left(\max_{\alpha \in \Sigma\cup\{0\}} \alpha(a)^2\right) \|Y\|^2\\
&=
\left(\max_{\alpha \in \Sigma^+} \alpha(a)^2\right) \|Y\|^2.\nonumber
\end{align}
Moreover, if we choose $\alpha_0\in \Sigma^+$ such that 
$\alpha_0(a)=\max_{\alpha \in \Sigma^+} \alpha(a)$
and $Y\in \gog^{\alpha_0}$, then the equality in \eqref{eq:norm0} holds.
Hence, we deduce that for every $g\in G$, there exists $Y$ contained in
single root space such that $\|Y\|=1$ and
\begin{equation}
\label{eq:g}
\|g\|_{\rm op} =\|a_g\|_{\rm op}= \|\Ad(a_g)Y\| = \max_{\alpha \in \Sigma} \alpha(a_g)= \max_{\alpha \in \Sigma^+} \alpha(a_g).
\end{equation}
This proves (i).

The claim (ii) is deduced from \eqref{eq:g}. Given $g=k_1ak_2 \in KA^+K$, we have 
$$
\|g\|_{\rm op}=\|\Ad(a)Y\|=\|\Ad(g) \Ad(k_2)^{-1}Y\|.
$$
Since $Y$ is contained in a single root space, the map $\Ad(Y)$ is nilpotent,
 and the map $\Ad(Z)$ with $Z=\Ad(k_2)^{-1}Y$ is also nilpotent. 
 We also have $\|Z\|=\|Y\|=1$. This implies (ii).

To prove (iii), we observe that since the metric $\rho_G$ is left-invariant, and $K$ is compact, there exists $c_2'>0$ such that for every $g=k_1a_gk_2\in KA^+K$, 
$$
\rho_G(a_g,e_G)-c'_2\le \rho_G(g,e_G)\le \rho_G(a_g,e_G)+c'_2.
$$
Since $A$ is abelian, there exists $c'_1\ge 1$ such that for all $X\in \hbox{Lie}(A)$,
$$
(c'_1)^{-1}\, \|X\|\le \rho_G(\exp(X),e_G)\le c'_1\, \|X\|.
$$
We also observe that the map
$$
X\mapsto \max_{\alpha\in\Sigma} (\log\circ \alpha\circ \exp)(X),\quad X\in\hbox{Lie}(A),
$$
defines a norm on $\hbox{Lie}(A)$. Hence, it follows from \eqref{eq:g}
that there exists $c_1''\ge 1$ such that
for all $X\in\hbox{Lie}(A)$,
$$
(c''_1)^{-1}\, \|X\|\le \log\|\exp(X)\|_{\rm op} \le c''_1\, \|X\|.
$$
Combining these estimates, we deduce (iii).

Now we proceed with the proof of (iv).
We introduce the operator norm on $\hbox{End}(\gog)$.
Since $\hbox{Ad}(X)$ is a nilpotent transformation, we obtain
\begin{equation}
\label{eq:ad}
\hbox{Ad}(\exp(X))Z=\exp(\hbox{Ad}(X))Z=\sum_{i=0}^{\dim(G)} \frac{1}{i!}\hbox{Ad}(X)^i Z
\end{equation}
for all $Z\in \gog$.
Hence, it follows that
$$
\|\exp(X)\|_{\rm op}\ll \max (1,\|\hbox{Ad}(X)\|)^{\dim(G)}
\ll \max(1,\|X\|)^{\dim(G)}.
$$
This proves one of the inequalities in (iv).
To prove the other inequality, we use Lemma \ref{l:norm}.
to obtain
$$
\|\exp(X)\|_{\rm op}=\|\exp(\hbox{Ad}(X))\|\gg \|\hbox{Ad}(X)\|.
$$
Also since $\exp(\hbox{Ad}(X))$ is unipotent, $\|\exp(\hbox{Ad}(X))\|\ge 1$.
Hence,
$$
\|\exp(X)\|_{\rm op}\gg \max(1,\|\hbox{Ad}(X)\|).
$$
Since $\gog$ is semisimple, the map $\Ad:\mathfrak{g}\to \hbox{End}(\mathfrak{g})$ is an embedding, and 
$$
\|\hbox{Ad}(X)\|\gg \|X\|\quad\hbox{for all $X\in\gog$.}
$$
Thus, we deduce the other inequality in (iv). 
\end{proof}

\subsection{Sobolev norms}\label{sec:sobolev}

Let $L$  be a connected Lie group and 
$\Gamma$ a lattice in  $L$. 
We consider the space $X = \Gamma \backslash L$  
equipped with the invariant probability measure $m$
and a Riemannian metric induced by a left-invariant Riemannian metric $\rho_L$ on $L$.
Let $\cC_c(X)$ denote the space of compactly 
supported continuous functions on $X$, endowed with the topology of uniform convergence on
compact subsets, and let $\cP(X)$ denote the space of Borel probability measures on $X$, which we view as a 
subspace of the dual space of $\cC_c(X)$. Even though $L$ acts naturally on $X$ from the right, we shall also 
write this action as a \emph{left} action, i.e., if $x = \Gamma s$, we set
\[
l\cdot x = \Gamma s l^{-1}, \quad \textrm{for $l \in L$}.
\]
We note that $L$ also acts on $\cC_c(X)$ and on $\cP(X)$ by
\[
(l \cdot \phi)(x) = \phi(l^{-1} \cdot x) \qand (l \cdot \nu)(\phi) = \nu(l^{-1} \cdot \phi),
\]
for $\phi \in\cC_c(X)$ and $\nu \in \cP(X)$. Finally, we denote by $\cC_c^\infty(X)$ the space of infinitely differentiable
compactly supported functions on $X$. \\

If $x = \Gamma s \in X$, we denote by $r(x)$ the 
\emph{injectivity radius} at $x$, i.e. the smallest $r > 0$ such that the quotient map $L \ra X$, restricted to 
a closed ball of $\rho_L$-radius $r$ around $s$, is injective. If $\Gamma$ is a co-compact lattice, this 
number stays uniformly away from zero. However, if $\Gamma$ is not co-compact, then $r(x)$ 
tends to zero as $x$ moves into the cusps of $X$. \\

Let $\mathfrak{l}=\hbox{Lie}(L)$.
We note that every $Y \in \gol$ gives rise to a first order differential operator $\cD_Y$ on $\cC^\infty_c(X)$ defined by
\[
(\cD_Y \phi)(x) = \lim_{t \ra 0} \frac{\phi(x \exp(tY)) - \phi(x)}{t}, \quad\quad \textrm{for $\phi \in \cC^\infty_c(X)$}.
\]
If we fix an ordered basis $\{Y_1,\ldots,Y_n\}$ for the Lie algebra $\gol$, then every element in the universal enveloping algebra $U(\gol)$
of $\gol$ can be written as (an ordered) linear combination of monomials in the basis elements 
of the form $Y_1^{m_1} \cdots Y_n^{m_n}$. Every 
such monomial $W$ gives rise to a differential operator by composition, i.e.,
\begin{equation}\label{eq:DDD}
\cD_W = \cD_{Y_1}^{m_1} \cdots \cD_{Y_n}^{m_n}.
\end{equation}
The {degree} $\deg(\cD_W)$ is defined as the sum $m_1 + \ldots + m_n$. \\

Following \cite{EMV}, for an integer $d \geq 1$ and $\phi \in\cC_c^\infty(X)$, the 
\emph{Sobolev norm of $\phi$ of degree $d$} is defined by
\begin{equation}
\label{def_Sobolev}
\cS_d(\phi) := \left(\sum_{\deg(\cD_W) \leq d} \int_X |r(x)^{-\kappa_d} (\cD_W \phi)(x)|^2 \, dm(x)\right)^{1/2},
\end{equation}
where $\kappa_d > 0$ are chosen appropriately, so that the following properties hold:
\begin{enumerate}
 \setlength\itemsep{0.5em}
\item[N1.] For sufficiently large $d$ and $\phi \in\cC_c^\infty(X)$,
\begin{equation}
\label{unilip}
\|\phi\|_\infty \ll_d  \cS_d(\phi).
\end{equation}

\item[N2.] For sufficiently large $d$, $\phi \in\cC_c^\infty(X)$, and $g\in G$,
\begin{equation}
\label{unilip00}
\|\phi - g \cdot \phi\|_\infty \ll_d \rho_G(g,e_G) \, \cS_d(\phi).
\end{equation}

\item[N3.] For sufficiently large $d$, an exponent $\sigma=\sigma(d) > 0$, $\phi \in\cC_c^\infty(X)$, and $g\in G$,
\begin{equation}
\label{gbnd}
\cS_d(g \cdot \phi) \ll_d \|g\|_{\rm op}^{\sigma} \, \cS_d(\phi).
\end{equation}

\item[N4.] There exists $r > 0$ such that for all sufficiently large $d$
and $\phi_1, \phi_2 \in\cC_c^\infty(X)$,
\begin{equation}
\label{dr}
\cS_d(\phi_1 \phi_2) \ll_{d} \cS_{d+r}(\phi_1) \, \cS_{d+r}(\phi_2).
\end{equation}
\end{enumerate}

The use of the term $r(x)^{-\kappa_d}$ in the definition \eqref{def_Sobolev}
is convenient in order to have statements which are uniform
on $\cC_c^\infty(X)$. If we restrict our attention 
the subalgebra of functions with
supports contained in a fixed compact subset of $X$,
then the norms $\cS_d$ are 
equivalent to the 
standard Sobolev norms.

\subsection{Mixing of higher orders and Wasserstein distances on couplings}

Let us now reformulate Theorems \ref{main1} and \ref{main1_coupl} in a way that better aligns
with the point of view taken  in the paper. We recall that 
$X = \Gamma \backslash L$ and that $m$ denotes the invariant probability measure on $X$. 

Given an integer $k \geq 2$, we write
$[k] = \{1,\ldots,k\}$, and given a subset $I \subset [k]$, we let $G_I$, $X_I$ and $m_I$ 
denote the direct product of $G$, $X$ and $m$ respectively, over the indices in $I$. We 
also write
\[
\Delta_I(G) = \{ (g,\ldots,g) \, : \, g \in G \} \subset G_I
\qand
\Delta_I(X) = \{ (x,\ldots,x) \, : \, x \in X \} \subset X_I,
\]
and we denote by $m_{\Delta_I(G)}$ the probability measure on $X_{I}$, which is the image of $m$ 
under the diagonal embedding of $X$ into $X_{I}$. 
With this notation, we note that if $g_1,\ldots,g_k \in G$ and $\phi_1,\ldots,\phi_k \in\cC^\infty_c(X)$, then
\[
m((g_1 \cdot \phi_1) \cdots (g_k \cdot \phi_k))
=
m_{\Delta_{[k]}(X)}((g_1\cdot \phi_1)\otimes\cdots \otimes (g_k\cdot \phi_k)).
\]
Since $m_{\Delta_{[k]}(X)}$ is a $\Delta_{[k]}(G)$-invariant $k$-coupling of $(X,m)$, Theorem \ref{main1} is a particular case of 
Theorem \ref{main1_coupl}.
 \\

Given $I = \{i_1,\ldots,i_l\} \subset [k]$, we let $\cC^\infty_c(X)_I$ denote 
the algebraic tensor product of the algebra $\cC^\infty_c(X)$ over the indices in $I$, that is to say, the subalgebra of 
$\cC_c(X_I)$ which is spanned by all finite sums of the form
\[
\sum_{j} \phi_{i_1\,j} \otimes \cdots \otimes \phi_{i_l\,j},
\]
where $\phi_{i_s\,j} \in \cC^\infty_c(X)$, $s = 1,\ldots,l$. Given an integer $d$, we define the \emph{projective tensor product} (or \emph{maximal cross-norm}) $S_{d,I}$ of the Sobolev norm $\cS_d$ of an element $\phi \in\cC_c^\infty(X)_I$ by
\[
\cS_{d,I}(\phi) := \inf\Big\{ \sum_{j} \cS_d(\phi_{i_1\,j}) \cdots \cS_d(\phi_{i_l\,j}) \Big\},
\]
where the infimum is taken over all possible ways to write $\phi$ as a finite sum of the form
\[
\phi = \sum_j \phi_{i_1\,j} \otimes \cdots \otimes \phi_{i_l\,j}, \quad\quad \textrm{with $\phi_{i_1\,j},\ldots,\phi_{i_l\,j} \in\cC_c^\infty(X)$}.
\]
We can readily extend the action $G_I \acts X_I$ to $\cC_c(X_I)$ and to $\cP(X_I)$, by 
\[
(g_I \cdot \phi)(x_I) = \phi(g_I^{-1} \cdot x_I) \qand (g_I \cdot \nu)(\phi) = \nu(g_I^{-1} \cdot \phi),
\]
for $\phi \in\cC_c(X_I)$, $\nu \in \cP(X_I)$ and $g_I \in G_I$. We note that the extended $G_I$-action on $\cC_c(X_I)$ preserves the subspace $\cC_c^\infty(X)_I$.\\

Let $\eta$ be a $k$-coupling of $(X,m)$.
Then for $g_1,\ldots,g_k \in G$ and $\phi_1,\ldots,\phi_k \in\cC^\infty_c(X)$,
\[
\eta((g_1 \cdot \phi_1)\otimes \cdots \otimes (g_k \cdot \phi_k))
=
((g_1,\ldots,g_k)^{-1} \cdot \eta)(\phi_1 \otimes \cdots \otimes \phi_k),
\]
and
\[
m(\phi_1) \cdots m(\phi_k) = m_{[k]}(\phi_1 \otimes \cdots \otimes \phi_k).
\]
Hence, Theorem \ref{main1_coupl}, which we wish to prove, can now be equivalently stated as: \\

\noindent {\it
For all $k\ge 2$ and sufficiently large $d$, there exists $\delta=\delta(k,d) > 0$ such that for all $\Delta_{[k]}(G)$-invariant couplings $\xi$ of $(X,m)$ and
$
g_{[k]} = (g_1,\ldots,g_k) \in G_{[k]},$
we have
\begin{equation}
\label{firsttry}
\sup \left\{ \left| (g_{[k]}^{-1} \cdot \xi)(\phi) - m_{[k]}(\phi) \right| \, : \, \phi\in\cC_c^\infty(X)_{[k]}\;\hbox{ with }\cS_{d,[k]}(\phi) \leq 1 \right\}
\ll_{d,k}  \mathfrak{M}(g_{[k]})^{-\delta}.
\end{equation}
}\\

This way of rewriting Theorem \ref{main1_coupl} motivates the following definition. If $Y$ is a locally compact metrizable
space, $\cA \subset\cC_c(Y)$ is a fixed linear subspace, and $M$ is a norm on $\cA$, then we define the 
\emph{Wasserstein distance} $\dist_M$ on the space $\cP(Y)$ of Borel probability measures on $Y$ by
\begin{equation}
\label{defW}
\dist_M(\mu,\nu) := \sup\big\{ |\mu(\phi) - \nu(\phi)| \, : \, \phi \in \cA\;\hbox{ with }  M(\phi) \leq 1 \big\},
\end{equation}
for $\mu, \nu \in \cP(Y)$. We stress that this is always a semi-distance (semi-metric), but only a metric if $\cA$ is 
dense in $\cC_c(Y)$, endowed with the topology of uniform convergence on compact subsets. We shall discuss properties of this semi-distance in more details in Section \ref{sec:wasserstein}. \\

Let us now adopt the following useful notation. If $I = \{i_1,\ldots,i_s\} \subseteq [k]$, then, given $\nu \in \cP(X_{[k]})$ 
and $g_{[k]} \in G_{[k]}$, we write $\nu_I$ for the push-forward of $\nu$ onto $X_I$, and we set
\[
g_I = (g_{i_1},\ldots,g_{i_s}) \in G_I.
\]
Furthermore, if $d$ is an integer, we set
\[
\dist_{d,I} := \dist_{\cS_{d,I}},
\]
which is indeed a metric on $\cP(X_I)$ since $\cC^\infty_c(X)_I$ is dense in $\cC_c(X_I)$. \\

With this notation, we can now further rewrite \eqref{firsttry} (the assertion of Theorem \ref{main1_coupl})
and reformulate Theorem \ref{main1_coupl} in terms of estimates on 
the distance $\dist_{d,[k]}$ as:\\

\noindent {\it
For all $k \ge 2$ and sufficiently large $d$, there exists $\delta=\delta(k,d) > 0$ such that for all $\Delta_{[k]}(G)$-invariant couplings $\xi$ of $(X,m)$ and 
$g_{[k]} = (g_1,\ldots,g_k) \in G_{[k]}$, we have 
\[
\dist_{d,[k]}(g_{[k]}^{-1} \cdot \xi,m_{[k]}) \ll_{d,k} \mathfrak{M}(g_{[k]})^{-\delta}.
\]
}\\


Now we state our main technical result --- Theorem \ref{main1_real}. In the next subsection, we show how to deduce Theorems \ref{main1} and \ref{main1_coupl} from it. 
We recall that $G$ is a connected semisimple Lie group with finite center, and we are assuming that $G$ is a closed subgroup of a connected Lie group $L$, and its action on a finite volume homogeneous space $X=\Gamma \backslash L$
has strong spectral gap. 
As before $m$ denotes the normalized invariant measure on $X$. 
Then by \cite[Cor.~3.2]{KS}, there exists $\delta_2 > 0$ such that for all sufficiently large integers $d$, functions $\phi_1,\phi_2\in\cC_c^\infty(X)$, and elements $g\in G$,
\begin{equation}
\label{KS}
| m((g \cdot \phi_1) \phi_2) - m(\phi_1) m(\phi_2) | \ll_d \|g\|_{\rm op}^{-\delta_2} \, \cS_{d}(\phi_1) \, \cS_d(\phi_2), 
\end{equation}
where $\| \cdot \|_{\rm op}$ is defined as in \eqref{def*}. 
While this estimate in \cite{KS} is stated in terms of the Riemanian distance $\rho_G$,
it follows from Lemma \ref{lemma1_outline_main1}(iii) that we also have the estimate of the form \eqref{KS}.

The following theorem inductively upgrades \eqref{KS} to an estimate for general $k$-couplings
of $(X,m)$ which are $\Delta_{[k]}(G)$-invariant. 

\begin{theorem}
\label{main1_real}
Let $g_{[k]} = (g_1,\ldots,g_k) \in G_{[k]}$ and
$\xi$ be a $\Delta_{[k]}(G)$-invariant $k$-coupling of $(X,m)$.
Suppose that there exist $F \geq 1$, $\tau > 0$ and an integer $d$ such that
\begin{equation}
\label{k-1ass}
\dist_{d,I}(g_I^{-1} \cdot \xi_I, m_I) \leq F\, \mathfrak{N}(g_{[k]})^{-\tau}, \quad \textrm{for all $I \subsetneq [k]$}.
\end{equation}
Then there exists $\gamma_k > 0$, which depends only on $k$, $d$ and $\delta_2$, such that
\begin{equation}
\label{conclus}
\dist_{d + r,[k]}(g_{[k]}^{-1} \cdot \xi,m_{[k]}) \ll_{d,k} \sqrt{F} \, \mathfrak{N}(g_{[k]})^{-\gamma_k \min(1,\tau)},
\end{equation}
where $r$ is as \eqref{dr}, and 
$$
\mathfrak{N}(g_{[k]}):=\min_{i\ne j} \|g_i^{-1}g_j\|_{\rm op}.
$$
\end{theorem}

We emphasise that the implied constant in \eqref{conclus} does \emph{not} depend on $\tau$, $\xi$, or the $k$-tuple $g_{[k]}$, and
if the bound \eqref{k-1ass} is uniform over all $\Delta_{[k]}(G)$-invariant $k$-couplings of $(X,m)$, then so is the bound \eqref{conclus}.

\subsection{Proof of Theorems \ref{main1} and \ref{main1_coupl} (assuming Theorem \ref{main1_real})}
\label{sec:proof_lie}

As we already noted Theorem \ref{main1} is a particular case of Theorem \ref{main1_coupl}
with $\xi = m_{\Delta_{[k]}(X)}$.

Let us fix $k \geq 2$ and a $k$-tuple $g_{[k]} = (g_1,\ldots,g_k) \in G_{[k]}$, and set
\[
q := \min_{i \neq j} \|g_i^{-1} g_j\|_{\rm op}.
\]
By 
Lemma \ref{lemma1_outline_main1}, there exist constants $c_1,c_2 > 0$ such that
\[
\rho_G(g_i,g_j) = \rho_G(e_G,g_i^{-1} g_j) \leq c_1 \log \|g_i^{-1} g_j\|_{\rm op} + c_2,
\]
and thus
\begin{equation}
\label{qrho0}
q 
\geq e^{-c_2/c_1} \, \mathfrak{M}(g_{[k]})^{1/c_1}.
\end{equation}
Since $\xi$ is a $\Delta_{[k]}(G)$-invariant $k$-coupling of $(X,m)$, \eqref{k-1ass} is trivial for $k = 2$.
We shall now argue by induction. Suppose that we have shown \eqref{k-1ass} for all $I \subsetneq [k]$, that
is to say, we have produced $d, \tau > 0$ and $F\ge 1$ such that
\[
\dist_{d,I}(g_I^{-1} \cdot \xi_I, m_{I}) \leq F\, q^{-\tau},
\]
for all $I \subsetneq [k]$. Theorem \ref{main1_real} now provides $\gamma_k > 0$ such that
\[
\dist_{d+r,[k]}(g_{[k]}^{-1} \cdot \xi,m_{[k]}) \ll_{d,k} \sqrt{F}\, q^{-\gamma_k \min(1,\tau)},
\]
and thus, by \eqref{qrho0}, 
\[
\dist_{d+r,[k]}(g_{[k]}^{-1} \cdot \xi,m_{[k]}) \ll_{d,k} \sqrt{F} \,
\mathfrak{M}(g_{[k]})^{-\gamma_k\min(1,\tau)/c_1}.
\]
Upon unwrapping the definition of $\dist_{d+r,[k]}$, and observing that the implicit constants are independent 
of $g_{[k]} \in G_{[k]}$, we see that this exactly means that
\begin{align*}
&|\xi((g_1 \cdot \phi_1)\otimes \cdots \otimes(g_k \cdot \phi_k)) - m(\phi_1) \cdots m(\phi_k)| \\
\ll_{d,k} &\,
\mathfrak{M}(g_{[k]})^{-\gamma_k\min(1,\tau)/c_1}\,
\cS_{d+r}(\phi_1) \cdots \, \cS_{d+r}(\phi_k),
\end{align*}
for all $g_1,\ldots,g_k \in G$ and $\phi_1, \ldots, \phi_k \in\cC_c^\infty(X)$.
This finishes the proof of 
Theorem \ref{main1_coupl}. \\

\subsection{A reduction of the proof of Theorem \ref{main1_real}} \label{sec:p2}
We retain the notation from Subsection \ref{Subsec:notation}. Let us fix $g_{[k]} = (g_1,\ldots,g_k) \in G_{[k]}$ and 
a $\Delta_{[k]}(G)$-invariant $k$-coupling $\xi$ of $(X,m)$. Throughout this subsection, we set $\eta = g_{[k]}^{-1} \cdot \xi$, and note that $\eta$ is again a $k$-coupling of $(X,m)$, but this time invariant under the subgroup 
\[
 g_{[k]}^{-1} \cdot \Delta_{[k]}(G) \cdot g_{[k]} \subset G_{[k]}.
\]
Define 
$$
Q := \max_{i \neq j} \|g_{i}^{-1} g_j\|_{\rm op}\quad \hbox{and}\quad q: = \mathfrak{N}(g_{[k]})=\min_{i \neq j} \|g_{i}^{-1} g_j\|_{\rm op}.
$$
We note that $Q\ge q\ge 1$.
Let us fix indices $i_1\ne i_s \in [k]$ such that $Q = \|g_{i_1}^{-1} g_{i_s}\|_{\rm op}$.
By Lemma \ref{lemma1_outline_main1}, 
there exists $Z \in \gog$ such that 
$\Ad(Z)$
is a nilpotent endomorphism of $\mathfrak{g}$, $\|Z\| = 1$, and
\[
Q = \|g_{i_1}^{-1} g_{i_s}\|_{\rm op} = \|\Ad(g_{i_1}^{-1} g_{i_s})Z\|.
\]
Then for all $i,j\in [k]$, we have 
$$
\|\Ad(g_i^{-1}g_j)Z\|\le \|g_i^{-1}g_j\|_{\rm op}\le \|\Ad(g_{i_1}^{-1} g_{i_s})Z\|.
$$
We can label the indices in $[k]$ so that
\[
\|\Ad(g_{i_1}^{-1}g_{i_s})Z\| \geq \|\Ad(g_{i_2}^{-1}g_{i_s})Z\| \geq \ldots \geq \|\Ad(g_{i_k}^{-1}g_{i_s})Z\|.
\]
Then, in particular, $\|\Ad(g_{i_k}^{-1}g_{i_s})Z\|\le \|\Ad(g_{i_s}^{-1}g_{i_s})Z\|=1$.
Changing the indexation, we may assume that 
\[
\|\Ad(g_{1}^{-1}g_{s})Z\| \geq \|\Ad(g_{2}^{-1}g_{s})Z\| \geq \ldots \geq \|\Ad(g_{k}^{-1}g_{s})Z\|.
\]
Let 
\[
Z_j = \frac{\Ad(g_{j}^{-1} g_{s})Z}{\|\Ad(g_{1}^{-1} g_{s})Z\|}\quad \hbox{and}
\quad w_j = \|Z_j\|, \quad\quad \textrm{for $j = 1,\ldots,k$}.
\]
Then
\begin{equation}
\label{normalization}
1 = w_1 \geq w_2 \geq \ldots \geq w_k\quad \hbox{and} \quad w_k \le Q^{-1}\le q^{-1}.
\end{equation}
Let us fix an index $1 \leq p \le k-1$, and write $[k] = I \sqcup J$, where
\[
I = [1,p] \qand J = [p+1,k].
\]
For every $j \in [k]$, we define the flow $h_j : \bR \times X \ra X$ by
\begin{equation}
\label{defhj}
h_j(t) \cdot x = \exp(tZ_j) \cdot x, \quad \textrm{for $x \in X$}.
\end{equation}
We also define the flows $h_I : \bR \times X_I \ra X_I$ and $h_J : \bR \times X_J \ra X_J$ by
\begin{equation}
\label{eq:hI}
h_I(t) \cdot x_I = (h_1(t) \cdot x_1,\ldots,h_p(t) \cdot x_p)
\end{equation}
and
\begin{equation}
\label{eq:hJ}
h_J(t) \cdot x_J = (h_{p+1}(t) \cdot x_{p+1},\ldots,h_k(t) \cdot x_k),
\end{equation}
as well as the joint flow $h : \bR \times X_{[k]} \ra X_{[k]}$ by
\[
h(t) \cdot (x_I,x_J) = (h_I(t) \cdot x_I, h_J(t) \cdot x_J).
\]
Since
$$
(Z_1,\ldots, Z_k)\in \hbox{Lie}\left(g_{[k]}^{-1} \cdot \Delta_{[k]}(G) \cdot g_{[k]}\right)
=\Ad(g_{[k]})^{-1}\left(\hbox{Lie}(\Delta_{[k]}(G))\right),
$$
it follows that $\eta$ is an $h$-invariant $k$-coupling of $(X,m)$.
Similarly, its marginals $\eta_I$ and $\eta_J$ on $X_I$ and $X_J$ are invariant under 
the flows $h_I$ and $h_J$ respectively. \\

Let us fix a large integer $d$ so that the Sobolev norms $N := \cS_{d}$ on $\cA=C_c^\infty(X)$ satisfy (cf. \eqref{unilip}, \eqref{unilip00}, and \eqref{gbnd})
\begin{align}
\|\phi\|_\infty &\ll_d  N(\phi), \label{eq:n1_0} \\
\|\phi-g\cdot \phi\|_\infty &\ll_{d} \rho_G(g,e) \, N(\phi)\quad \hbox{for all $g \in G$,} \label{eq:n2_0}\\
N(g \cdot \phi) &\ll_{d} \|g\|_{\rm op}^{\sigma}\, N(\phi)\quad\hbox{for all $g\in G$ and some $\sigma =\sigma(d)> 0$}.\label{eq:n3_0}
\end{align}
We denote by $\cA_I$ and $\cA_J$ the algebraic tensor product of $\cA$ over the indices
in $I$ and $J$ respectively, and we let $N_I$ and $N_J$ denote the norms on $\cA_I$ and $\cA_J$ respectively which are 
projective tensor products of $N$. \\

Let us now list three important properties of the flows $h_I$ and $h_J$ that will be crucial in our analysis. In
all three lemmas, the index $p$ (and hence the partition $[k] = I \sqcup J$) will be fixed. We shall reduce the
proof of Theorem \ref{main1_real} to a general inequality for couplings which are invariant under a suitable
flow. This inequality (which is valid in a more general context as well) will be established in Section 
\ref{PrfABCDT} below. 

\begin{lemma}
\label{lemma1}
There exist $A \geq 1$ and $a > 0$,
depending only on $d$ and $k$, such that for all $t \in \mathbb{R}$ and $\phi_I \in \cA_I$,
$$
N_I(h_I(t) \cdot \phi_I) \leq A \max(1,|t|)^a \, N_I(\phi_I).
$$
\end{lemma}

\begin{proof}
Pick $\phi_I \in \cA_I$ and write it as a finite sum of the form
\[
\phi_I = \sum_{i} \phi_{1\,i} \otimes \cdots \otimes \phi_{p\,i},
\]
for some $\phi_{ji} \in \cA$. For every $t\in\mathbb{R}$, we have
\[
h_I(t) \cdot \phi_I = \sum_i (h_1(t) \cdot \phi_{1\,i}) \otimes \cdots \otimes(h_p(t) \cdot \phi_{p\,i}),
\]
and thus
\[
N_I(h_I(t) \cdot \phi_I) \leq \sum_i N(h_1(t) \cdot \phi_{1\,i}) \cdots N(h_{p}(t) \cdot \phi_{p\,i}).
\]
By \eqref{eq:n3_0} and Lemma \ref{lemma1_outline_main1}(iv), we have 
\begin{align*}
N(h_j(t) \cdot \phi_{j\,i})&\ll_{d} \|h_j(t)\|_{\rm op}^{\sigma}\, N(\phi_{j\,i})
\ll \max(1,\|tZ_j\|)^{\sigma\dim(G)}\, N(\phi_{j\,i})\\
&=\; \max(1,w_j |t|)^{\sigma\dim(G)}\, N(\phi_{j\,i}).
\end{align*}
Since $w_j \leq 1$ for all $j$,
\begin{eqnarray*}
N_I(h_I(t)  \cdot \phi_I) 
&\ll_{d,p} & 
\sum_i \Big( \prod_{j=1}^p \max(1,w_j |t|)^{\sigma\dim(G)} \Big) \, N(\phi_{1\,i}) \cdots N(\phi_{p\,i}) \\
&\leq & 
 \max(1,|t|)^{k \sigma\dim(G)} \, \sum_i N(\phi_{1\,i}) \cdots N(\phi_{p\,i}),
\end{eqnarray*}
This implies that
$$
N_I(h_I(t)  \cdot \phi_I) 
\ll_{d,k}  \max(1,|t|)^{a} \, N_I(\phi_{I})
$$
with $a=k \sigma\dim(G)$, which finishes the proof.
\end{proof}

\begin{lemma}
\label{lemma2}
There exists $B \geq 1$, depending only on $d$ and $k$, such that for all $t\in \mathbb{R}$ and $\phi_I \in \cA_I$,
$$
|m_I((h_I(t) \cdot \phi_I) \phi_I) - m_I(\phi_I)^2| \leq B \max(1,w_p |t|)^{-\delta_2} \, N_I(\phi_I)^2,
$$
where $\delta_2$ is as in \eqref{KS}.
\end{lemma}

\begin{proof}
Pick $\phi_I \in \cA_I$ and write it as a finite sum of the form
\[
\phi_I = \sum_i \phi_{1\,i} \otimes \cdots \otimes \phi_{p\,i},
\]
for some $\phi_{s\,i}\in\cA$.
For all $t\in \mathbb{R}$, we have
\[
m_I((h_I(t) \cdot \phi_I) \phi_I) = \sum_{i,j} m((h_1(t) \cdot \phi_{1\,i})\phi_{1\,j})  \cdots m((h_p(t) \cdot \phi_{p\,i})\phi_{p\,j}),
\]
and
\[
m_I(\phi_I)^2 = \sum_{i,j} m(\phi_{1\,i})m(\phi_{1\,j}) \cdots m(\phi_{p\,i})m(\phi_{p\,j}).
\]
The difference $m_I((h_I(t) \cdot \phi_I)\phi_I) - m_I(\phi_I)^2$ can be written as a finite sum of terms of the form
\[
m((h_1(t) \cdot \phi_{1\,i})\phi_{1\,j}) \cdots m((h_p(t) \cdot \phi_{p\,i})\phi_{p\,j})
- 
m(\phi_{1\,i})m(\phi_{1\,j}) \cdots m(\phi_{p\,i})m(\phi_{p\,j}).
\]
Each such term can be written as a sum of $p$ terms of the form
\begin{equation}
\label{Tlij}
(m((h_l(t) \cdot \phi_{l\,i})\phi_{l\,j}) - m(\phi_{l\,i})m(\phi_{l\,j}))T_{lij},
\end{equation}
where 
\[
|T_{lij}| \leq \prod_{s \neq l} \|\phi_{s\,i}\|_{\infty} \|\phi_{s\,j}\|_\infty.
\]
By \eqref{KS}, we have for some $\delta_2 > 0$,
\[
| m((h_l(t) \cdot \phi_{l\,i})\phi_{l\,j}) - m(\phi_{l\,i})m(\phi_{l\,j})| \ll_{d}  \|\exp(tZ_l)\|^{-\delta_2}_{\rm op} \, N(\phi_{l\,i})N(\phi_{l\,j}),
\]
for all indices $l,i,j$ and $t\in\mathbb{R}$. We recall that $Z$, and thus $Z_l$ (being the image of $Z$ under an adjoint operator), 
is such that $\Ad(Z_l)$ is nilpotent. Hence, the assertion (iv) in Lemma \ref{lemma1_outline_main1} shows that
\[
\|\exp(tZ_l)\|_{\rm op} \gg \max(1,|t|\, \|Z_l\|)=\max(1,|t|w_l).
\] 
We recall that $w_p \leq w_l$ for all $l \in I$, so that
\begin{align*}
| m((h_l(t) \cdot \phi_{l\,i})\phi_{l\,j}) - m(\phi_{l\,i})m(\phi_{l\,j})| &\ll_{d}
\; \max(1,w_p |t|)^{-\delta_2} \, N(\phi_{l\,i})N(\phi_{l\,j}).
\end{align*}
In view of  \eqref{eq:n1_0}, we conclude that each term of the form \eqref{Tlij} satisfies
\[
\ll_d \max(1,w_p |t|)^{-\delta_2} \prod_{l=1}^p N(\phi_{l\,i}) N(\phi_{l\,j}),
\]
uniformly over all $i,j$. Hence,
$$
m_I((h_I(t) \cdot \phi_I)\phi_I) - m_I(\phi_I)^2
\ll_{d,k} \max(1,w_p|t|)^{-\delta_2}\, \left(\sum_{i} \prod_{l=1}^p N(\phi_{l\,i})\right)^2,
$$
and
$$
m_I((h_I(t) \cdot \phi_I)\phi_I) - m_I(\phi_I)^2
\ll_{d,k} \max(1,w_p|t|)^{-\delta_2}\, N_I(\phi_I)^2.
$$
This proves the claim of the lemma.
\end{proof}

\begin{lemma}
\label{lemma3}
There exists $C \geq 1$, depending only on $d$ and $k$, such that for all $t\in\mathbb{R}$ and $\phi_J \in \cA_J$,
$$
\|h_J(t) \cdot \phi_J - \phi_J\|_\infty \leq C\, w_{p+1} |t|  \, N_J(\phi_J).
$$
\end{lemma}

\begin{proof}
Pick $\phi_J \in \cA_J$ and write it as a finite sum of the form
\[
\phi_J = \sum_j \phi_{p+1\,j} \otimes \cdots \otimes \phi_{k\,j},
\]
for some $\phi_{l\,j} \in \cA$. We can write each term in the difference $h_J(t) \cdot \phi_J - \phi_J$ as a telescoping 
sum of the $k-p$ terms of the form
\[
\phi_{p+1\,j} \otimes \cdots \otimes \phi_{l-1\,j}\otimes (h_{l}(t) \cdot \phi_{l\,j} - \phi_{l\,j}) \otimes h_{l+1}(t) \cdot \phi_{l+1\,j} \otimes \cdots h_k(t) \cdot \phi_{k\,j}.
\]
Hence,
\[
\|h_J(t) \cdot \phi_J - \phi_J\|_\infty \leq \sum_j \sum_{l=p+1}^k \|h_l(t) \cdot \phi_{l\,j} - \phi_{l\,j}\|_\infty \, \prod_{i \neq l} \|\phi_{i\,j}\|_\infty
\]
Using \eqref{eq:n1_0} and \eqref{eq:n2_0}, we see that uniformly on $l,j$,
\[
\|h_l(t) \cdot \phi_{l\,j} - \phi_{l\,j}\|_\infty \, \prod_{i \neq l} \|\phi_{i\,j}\|_\infty
\ll_{d} \rho_G(\exp(tZ_l),e_G)
\, \prod_{i=p+1}^{k} N(\phi_{i\,j}),
\]
Since
\[
\rho_G(\exp(t Z_l),e_G) \ll \|t Z_l\| = |t|w_l \leq |t| w_{p+1}, \quad \textrm{for all $t\in \mathbb{R}$ and $l \in J$},
\]
we conclude that 
\[
\|h_J(t) \cdot \phi_J - \phi_J\|_\infty \ll_{d,k} |t| w_{p+1} \sum_j N(\phi_{p+1\,j}) \cdots N(\phi_{k\,j}),
\]
which implies the lemma.
\end{proof}

\subsection{Finishing the proof of Theorem \ref{main1_real} (assuming Proposition \ref{mainineq})}

We recall the setting of Theorem \ref{main1_real}: $\xi$ is a $\Delta_{[k]}(G)$-invariant $k$-coupling of $(X,m)$, and 
$\eta = g_{[k]}^{-1} \cdot \xi$ for some $g_{[k]} = (g_1,\ldots,g_k) \in G_{[k]}$. The latter measure is
invariant under the flow $h$, which is defined using an elements $Z_j \in \gog$
such that $\Ad(Z_j)$ is nilpotent, 
whose norms $w_j = \|Z_j\|$ satisfy
\[
1 = w_1 \geq w_2 \geq \ldots \geq w_k \quad\hbox{and}\quad w_k\le q^{-1},
\]
where $q = \min_{i \neq j} \|g_i^{-1} g_j\|_{\rm op}$. 
We recall that $\|g\|_{\rm op}\ge 1$ for all $g\in G$, so that $q\ge 1$.\\

Let $d$ and $r$ be integers such that the Sobolev norms $M := \cS_{d}$ and $N := \cS_{d+r}$ on $\cA=C_c^\infty(X)$ satisfy
\begin{align}
\|\phi\|_\infty &\ll_d M(\phi)\ll_{d,r} N(\phi), \label{eq:n1} \\
\|\phi-g\cdot \phi\|_\infty &\ll_{d,r} \rho_G(g,e) \, N(\phi)\quad \hbox{for all $g \in G$,} \label{eq:n2}\\
N(g \cdot \phi) &\ll_{d,r} \|g\|_{\rm op}^{\sigma}\, N(\phi)\quad\hbox{for all $g\in G$ and some $\sigma =\sigma(d)> 0$},\label{eq:n3}\\
M(\phi_1 \cdot \phi_2) &\ll_{d,r} N(\phi_1) N(\phi_2) \label{eq:n4} 
\end{align}
(see \eqref{unilip}, \eqref{unilip00}, \eqref{gbnd}, and \eqref{dr}).\\

We have shown in Lemmas \ref{lemma1}, \ref{lemma2}, and \ref{lemma3} above that 
given a non-trivial decomposition $[k]=I\sqcup J$ where $I=[1,p]$ and $J=[p+1,k]$, 
there are constants $A, B, C$ and $a, b > 0$ (which 
are independent of $p$), such that for all $t\in\mathbb{R}$ and for all $\phi_I \in \cA_I$ and $\phi_J \in \cA_J$,
\begin{equation}
\label{polgr}
N_I(h_I(t) \cdot \phi_I) \leq A\, \max(1,|t|)^a \, N(\phi_I),
\end{equation}
\begin{equation}
\label{polmix}
|m_I((h_I(t) \cdot \phi_I) \phi_I) - m_I(\phi_I)^2| \leq B\, \max(1,w_p |t|)^{-b} \, N_I(\phi_I)^2,
\end{equation}
\begin{equation}
\label{polmix}
\|h_J(t) \cdot \phi_J - \phi_J\|_\infty \leq C\, w_{p+1} |t|  \, N_J(\phi_J).
\end{equation}
Assuming the bounds \eqref{polgr}--\eqref{polmix},
we establish the following useful inequality.
A general form of this inequality  will be proved in Section \ref{PrfABCDT}
(see Proposition \ref{mainest} below). 

\begin{proposition}
\label{mainineq}
For every fixed $1 \leq p < k$ and $T \in [w_p^{-1},w_{p+1}^{-1}]$, we have
\begin{equation}
\dist_{N_{[k]}}(\eta,m_{[k]}) \ll_{d,r,k} \max((\mathcal{M}T^a)^{1/2},(w_p T)^{-b/2},w_{p+1} T),
\end{equation}
where
\[
\mathcal{M} := \max(\dist_{M_I}(\eta_I,m_I),\dist_{M_J}(\eta_J,m_J)).
\]
\end{proposition}

Next, we show
how to deduce Theorem \ref{main1_real} from Proposition \ref{mainineq}.

In order to finish the proof of Theorem \ref{main1_real},
we need to solve first the following problem. Given $q \geq 1$, we wish to minimize (over $p$ and $T$ for
which $T\in [w_p^{-1},w_{p+1}^{-1}]$) the 
expression
\begin{equation}
\label{maxexp}
\max((\mathcal{M}T^a)^{1/2}, (w_p T)^{-b/2}, w_{p+1} T)\le \sqrt{F} \max(q^{-\tau/2} T^{a/2}, (w_p T)^{-b/2}, w_{p+1} T),
\end{equation}
where $w_1,\ldots,w_k$ is a fixed sequence which satisfies
\[
1 = w_1 \geq w_2 \geq \ldots \geq w_k\quad\hbox{and}\quad w_k\le q^{-1}.
\]
We stress that we do not want the attained bound to depend on the particular choice of this sequence. \\

To solve this problem, let us choose  
$$
\delta = \gamma_0 \min(1,\tau), \quad \textrm{where $\gamma_0 = \min(1/k,1/(2ak))$},
$$
and note that
\[
k \delta \le 1 \qand 2ak \delta 
\le \tau.
\]
In particular, all of the $k$ points $q^{-\delta i}$, $0\le i\le k-1$, lie between $w_k$ and $w_1$. Hence, by the Pigeonhole Principle, there will be at least two consecutive points 
$q^{-\delta (i+1)}$ and $q^{-\delta i}$ for some $0\le i\le k-2$ which will end up in one and the same 
of the $(k-1)$ intervals
\[
[w_k,w_{k-1}), \ldots, [w_3,w_2), [w_2,w_1].
\]
Namely, for some $p=1,\ldots,k-1$ and $i=0,\ldots, k-2$,
\begin{equation}
\label{eq:b}
w_{p+1} \leq q^{-(i+1)\delta} \leq q^{-i \delta} \leq w_p
\end{equation}
For this particular $i$, we set $T = q^{(i+1/2)\delta}$, which clearly satisfies 
\[
w_p T = w_p q^{(i+1/2)\delta}  \geq q^{-i \delta} q^{(i+1/2)\delta}  = q^{\delta/2} \geq 1,
\]
and 
\[
w_{p+1} T \leq q^{-(i+1)\delta} q^{(i+1/2)\delta}  = q^{-\delta/2} \leq 1.
\]
Using \eqref{eq:b}, we deduce that that the expression \eqref{maxexp} is bounded from above by
\[
\sqrt{F} \max(q^{-(\tau - a (i+1/2)\delta)/2}, q^{-b \delta/4},q^{-\delta/2}).
\]
Since 
\[
\tau - a (i+1/2)\delta \geq \tau - a k\delta \geq \tau/2,
\]
by our choice of $\delta$, we conclude that \eqref{maxexp} is bounded from above by $q^{-\sigma}$, where
\[
\sigma = \min(\tau/2,b \delta/4,\delta/2) \geq \gamma_k \min(1,\tau),
\]
where $\gamma_k >0$ depends only  on $k$, $a$, and $b$. 
We note that, more precisely, $\gamma_k \gg 1/k$.\\

We have shown that for $q=\mathfrak{N}(g_{[k]})$, 
if there are constants $F\ge 1$ and $\tau > 0$ such that
\[
\max(\dist_{M_I}(\eta_I,m_I),\dist_{M_J}(\eta_J,m_J)) \leq F\, q^{-\tau},
\]
then 
$$
\dist_{N}(\eta,m_{[k]}) \ll_{d,k} \sqrt{F}\, q^{-\gamma_k \min(1,\tau)},
$$
where $\gamma_k$ depends only on 
$k, a$ and $b$. This finishes the proof of Theorem \ref{main1_real} (modulo the proof of Proposition \ref{mainineq}).
In Section \ref{PrfABCDT}, we prove a general version of Proposition \ref{mainineq}
(see Proposition \ref{mainest}).

\section{Approximate configurations in lattices}
\label{sec:conf}

Let $G$ be a connected semisimple Lie group with finite centre having no compact factors.
Let $\Gamma$ be an irreducible lattice in $G$. We fix a left-invariant Riemannian 
metric $\rho_G$ on $G$ which is bi-invariant under a fixed maximal compact subgroup of $G$.
The aim of this section is to prove Corollary \ref{main1_cor}.
Namely, we want to show that given $\eps>0$ 
and a $k$-tuple $(g_1,\ldots,g_k)\in G^k$ with sufficiently large
$$
w(g_1,\ldots,g_k):=\min_{i\ne j} \rho_G(g_i,g_j),
$$
there exist a $k$-tuples $(\gamma_1,\ldots, \gamma_k)\in \Gamma^k$
and an element $g\in G$ such that 
$$
\rho_G(g_i,g\, \gamma_i)<\eps\quad \quad\hbox{for all $i=1,\ldots,k$.}
$$
To construct such a tuple in $\Gamma^k$, we apply Theorem \ref{main1}
to the action of $G$ on the space $X=\Gamma\backslash G$ equipped with the invariant
probability measure $m$. 
Since $G$ has no compact factors, it is known (see, for instance, \cite[p.~285]{KSar})
that this action has strong spectral gap.
Hence, by Theorem \ref{main1}, 
for suitable $d\in \mathbb{N}$ and $\delta>0$,
\begin{equation}
\label{eq:mmm}
m((g_1\cdot \phi_1)\cdots (g_k\cdot \phi_k))=m(\phi_1)\cdots m(\phi_k)
+O_{d,k}\left(e^{-\delta\, w(g_1,\ldots,g_k)}\, \cS_d(\phi_1)\cdots \cS_d(\phi_k) \right)
\end{equation}
for all functions $\phi_1,\ldots \phi_k\in \cC_c^\infty(X)$.
We apply this estimate to suitably chosen $\phi_i$'s.
We choose the Haar measure $m_G$ on $G$ so that
$$
\int_G \Phi\, dm_G=
\int_{\Gamma\backslash G} \left( \sum_{\gamma\in\Gamma} \Phi(\gamma g)\right) dm(\Gamma g)
\quad\quad \hbox{for all $\Phi\in \cC_c(G)$.}
$$
Let $\Phi_\eps \in \cC_c^\infty(G)$ be 
a non-zero non-negative function such that 
$\supp(\Phi_\eps)\subset B_\eps(e)$,
$\| \Phi_\eps \|_1=1$, and
$\|\mathcal{D}\Phi_\eps\|_2\ll \eps^{-\theta}$
with some $\theta=\theta(d)>0$
for all differential operators $\mathcal{D}$ as in \eqref{eq:DDD} with $\deg(\mathcal{D})\le d$.
Let
$$
\phi_\eps(\Gamma g):=\sum_{\gamma\in\Gamma} \Phi_\eps(\gamma g).
$$
Then $m(\phi_\eps)=1$ and $\cS_d(\phi_\eps)\ll_d \eps^{-\theta}$.
Hence, we deduce from \eqref{eq:mmm} that 
$$
m((g_1\cdot \phi_\eps)\cdots (g_k\cdot \phi_\eps))=1+O_{d,k}\left(e^{-\delta\, w(g_1,\ldots,g_k)}\, \eps^{-k\theta} \right).
$$
If we take $c>k\theta/\delta$ and assume that
$w(g_1,\ldots,g_k)\ge c\log(1/\eps)$, then it follows that for all sufficiently
small $\eps>0$, we have 
$$
m((g_1\cdot \phi_\eps)\cdots (g_k\cdot \phi_\eps))>0.
$$
Finally, we observe that
\begin{align*}
m((g_1\cdot \phi_\eps)\cdots (g_k\cdot \phi_\eps))=
\int_{\Gamma\backslash G} \left(\sum_{(\gamma_1,\ldots,\gamma_k)\in \Gamma^k} \Phi_\eps(\gamma_1gg_1)\cdots \Phi_\eps(\gamma_kg g_k)\right) \, dm(\Gamma g),
\end{align*}
so that it follows that there exist
$(\gamma_1,\ldots,\gamma_k)\in \Gamma^k$ and $g\in G$ such that 
$$
\gamma_1gg_1,\ldots, \gamma_kgg_k\in \supp(\Phi_\eps)\subset B_\eps(e).
$$
Then for all $i=1,\ldots, k$, 
$$
\rho_G(g_i, g^{-1}\gamma_i^{-1})=\rho_G(\gamma_i g g_i,e_G)<\epsilon.
$$
This implies Corollary \ref{main1_cor}.

\section{Higher-order correlations for $S$-alebraic groups}
\label{sec:salg}

\subsection{Preliminaries}

Let $\bG$ be a simply connected absolutely simple algebraic groups
defined over a number field $F$, $S$ is a finite set of places of $F$,
and $G:=\prod_{v\in S} G_v$ where $G_v=\bG(F_v)$. 
For each $v\in S$, we fix a maximal $F_v$-split torus
$\bA_v$ of $\bG$ and set $A_v=\bA_v(F_v)$. We denote by $\gog_v$ the Lie algebra of the $p$-adic
Lie group $G_v$ and by $\Sigma_v\subset \bA^*_v$
the root system with respect to the adjoint action of the torus $\bA_v$ on the Lie algebra of $\bG$.
Then there is the root space decomposition $\gog_v$:
\begin{equation}
\label{eq:r0}
\gog_v=\gog_v^0+\bigoplus_{\alpha\in \Sigma_v} \gog_v^\alpha,
\end{equation}
where $\gog_v^0$ is the centraliser of $\hbox{Lie}(A_v)$ in $\gog_v$, and 
$$
\gog_v^\alpha:=\{Z\in \gog:\, \Ad(a)Z=\alpha(a)Z\quad\hbox{ for all $a\in A_v$}\}.
$$
It will be convenient to write succinctly 
$$
\gog:=\bigoplus_{v\in S} \gog_v\quad\hbox{and}\quad A:=\prod_{v\in S} A_v.
$$
Then we have the decomposition
\begin{equation}
\label{eq:ggg0}
\gog=\gog^0+\bigoplus_{\alpha\in \Sigma} \gog^\alpha,
\end{equation}
where $\Sigma=\cup_{v\in S} \Sigma_v$, and $\gog^\alpha$'s are root space for the action of $A$ on $\gog$.
These are precisely the roots spaces in the decompositions \eqref{eq:r0}.

We choose a system $\Sigma_v^+\subset \Sigma_v$ of positive roots,
and define the closed Weyl chamber
$$
A_v^+:=\{a\in A_v\,:\,\, |\alpha(a)|_v\ge 1 \quad\hbox{for all $\alpha\in \Sigma_v^+$}\}.
$$
There exists a good maximal compact subgroup
$K_v$ of $\bG(F_v)$ and a finite subset $\Omega_v$ of the centraliser
of $A_v$ in $\bG(F_v)$ such that the Cartan decomposition
\begin{equation}
\label{eq:cartan_padic}
\bG(F_v)=K_v A_v^+ \Omega_v K_v
\end{equation}
holds (\cite{BT}, \cite{Ti}). We write succinctly 
$$
K:=\prod_{v\in S} K_v,\quad\quad A^+:=\prod_{v\in S} A^+_v,\quad\quad \Omega:=\prod_{v\in S} \Omega_v.
$$
Then we have the decomposition
$$
G=KA^+\Omega K,
$$
which is an analogue of the real Cartan decomposition \eqref{eq:cartan}.

We fix a norm $\|\cdot \|$ on $\gog$ which is the maximum of fixed norms on $\gog_v$,
and define a sub-multiplicative function $\|\cdot\|_{\rm op}$ on $G$ by
\begin{equation}
\label{eq:norm00}
\|g\|_{\rm op}:= \max \{ \|\Ad(g)Z\|:\, Z\in\gog\; \hbox{ with $\|Z\|=1$}\}.
\end{equation}
This definition is similar to the definition \eqref{def*}.
One can also check as before that $\|g\|_{\rm op}\ge 1$ for all $g\in G$.

The following lemma is an analogue of Lemma \ref{lemma1_outline_main1}.

\begin{lemma}
	\label{lemma1_outline_main1_padic}
	\begin{enumerate}
		\setlength\itemsep{0.5em}
		
		\item[(i)] There exists $c_0\ge 1$ such that for every
		 $g= k_1 a \omega k_2 \in K A^+ \Omega K$,
		$$
		c_0^{-1}\, \left(\max_{\alpha \in \Sigma^{+}} \alpha(a)\right)\le \|g\|_{\rm op} \le c_0\, \left(\max_{\alpha \in \Sigma^{+}} \alpha(a)\right). 
		$$
		
		\item[(ii)] For every $g\in G$, there exists $Z\in \gog_v$ for some $v\in S$
		such that $\Ad(Z)$ is nilpotent, $\|Z\|=1$, and 
		$$
		\|g\|_{\rm op}\le c_0\,\|\Ad(g)Z\|.
		$$
		
		\item[(iii)] There exist constants $c_1,c_2\ge 1$ such that 
		\[
		c_2^{-1} \|g\|_{\rm op}^{c_1^{-1}} \leq \hbox{\rm H}(g) \leq c_2 \|g\|^{c_1}_{\rm op}
		\] 
		for all $g \in G$.

		\item[(iv)] There is a constant $c_3 \geq 1$ such that
		for every $X \in \gog_v$, $v\in S$, such that  
		$\hbox{\rm Ad}(X)$ is nilpotent, we have
		\[
		c_3^{-1}\, \max(1,\|X\|) \leq \|\exp(X)\|_{\rm op} \leq c_3\, \max(1,\|X\|)^{\dim(G)}.
		\]
	\end{enumerate}
\end{lemma}

\begin{proof}
	We first observe that by compactness there exists $c'_0\ge 1$ such that
	for every $g\in (K\cup \Omega K)^{\pm 1}$ and $Z\in\gog$,
	$$
	(c_0')^{-1}\, \|Z\|\le \|\Ad(g)Z\|\le c_0'\, \|Z\|.
	$$
	This, in particular, implies that for every $g=k_1a\omega k_2\in KA^+\Omega K$,
	\begin{equation}
	\label{eq:a_Salg}
	(c'_0)^{-2}\, \|a\|_{\rm op} \le \|g\|_{\rm op}\le (c_0')^2\, \|a\|_{\rm op}.
	\end{equation}
	Similarly, we also have
	$$
	\|a\|_{\rm op} \ll \hbox{H}(g)\ll  \|a\|_{\rm op}.
	$$
	Hence, it is sufficient to prove (i) and (iii) for $g\in A^+$.
	
	Decomposing $Y\in \gog$ with respect to the decomposition \eqref{eq:ggg0}
	as $Y=\sum_{\alpha\in\Sigma\cup\{0\}} Y_\alpha$,
	we deduce that for every $a\in A^+$,
	\begin{equation}
	\label{eq:nnn}
	\|\hbox{Ad}(a)Y\|\le\max_{\alpha\in \Sigma\cup \{0\}} |\alpha(a)|\,\|Y_\alpha\|\le 
	\left(\max_{\alpha\in \Sigma\cup \{0\}} |\alpha(a)|\right) \|Y\|
	=\left(\max_{\alpha\in \Sigma^+} |\alpha(a)|\right) \|Y\|.
	\end{equation}
	If we choose $\alpha_0\in\Sigma^+$ such that $\alpha_0(a)=\max_{\alpha\in \Sigma^+} \alpha(a)$ and $Y\in \gog^{\alpha_0}$, then the equality in \eqref{eq:nnn} holds.
	This implies that given any $a\in A^+$, there exists $Y\in \gog$ 
	contained in a single root space such that
	$\|Y\|=1$ and
	\begin{equation}
	\label{eq:a7}
	\|a\|_{\rm op}=\|\Ad(a)Y\|=\max_{\alpha\in \Sigma^+} |\alpha(a)|.
	\end{equation}
	This completes the proof of (i).

	To prove (ii), we observe that it follows from \eqref{eq:a_Salg} and \eqref{eq:a7}
	that for $g=k_1a\omega k_2\in KA^+\Omega K$,
	\begin{align*}
	\|g\|_{\rm op}&\le (c_0')^2\, \|a\|_{\rm op}=(c_0')^2\, \|\Ad(a)Y\|\le (c_0')^3\, \|\Ad(k_1a\omega k_2) \Ad(\omega k_2)^{-1}Y\|\\
	&= (c_0')^3 \|\Ad(\omega k_2)^{-1}Y\|\, \|\Ad(g) Z\|\le (c_0')^4\, \|\Ad(g) Z\|,
	\end{align*}
	where $Z=\frac{\Ad(\omega k_2)^{-1}Y}{\|\Ad(\omega k_2)^{-1}Y\|}$.
	Since $Y$ is contained in a single root spaces $\gog^\alpha_v$ for some $v\in S$,
	the map $\Ad(Y)$ is nilpotent. This also implies that
	$Z\in \gog_v$, and $\Ad(Z)$ is nilpotent. Hence, (ii) is proved.
	
	Now we proceed with the proof of (iii).
	Let us fix the set of simple roots $\Pi_v\subset \Sigma^+_v$, $v\in S$.
	Then every $\alpha\in \Sigma_v^+$ can be expressed as a  product of simple roots
	with non-negative exponents, so that
	there exists $c'_1\ge 1$ such that for every $a\in A^+$,
	\begin{equation}
	\label{eq:aa}
	\left(\max_{\alpha\in \Pi} |\alpha(a)|\right)\le \|a\|_{\rm op}\le \left( \max_{\alpha\in \Pi} |\alpha(a)|\right)^{c_1'},
	\end{equation}
	where $\Pi=\cup_{v\in S}\Pi_v$ is considered as a subset of the set of characters of $A$.

	We observe that $\|\cdot\|_v\gg 1$ on $G_v$ for $v\in S$, so that 
	$$
	\max_{v\in S} \|g_v\|_v\ll \hbox{H}(g)\ll \left(\max_{v\in S} \|g_v\|_v\right)^{|S|}
	\quad \hbox{ for $g=(g_v)_{v\in S}\in G$.}
	$$
	We consider the representation of $\bG$ defined by
	the embedding $\bG \subset \hbox{GL}_n$. Let $\Phi_v$ denote the set of weights of this representation with respect to the torus $\bA_v$. Since $\bA_v$ is split over $F_v$,
	the action of $A_v=\bA_v(F_v)$ on $F_v^n$ is completely reducible.
	This implies that for $a_v\in A_v$, $v\in S$,
	$$
	\max_{\chi\in \Phi_v} |\chi(a_v)|\ll \|a_v\|_v\ll \max_{\chi\in \Phi_v} |\chi(a_v)|.
	$$
	Hence, there exists $c_2\ge 1$ such that for every $a\in A$,
	\begin{equation}
	\label{eq:aaa}
	c_2^{-1} \left(\max_{\chi\in \Phi} |\chi(a)|\right)\le \hbox{H}(a)\le c_2 \left(\max_{\chi\in \Phi} |\chi(a)|\right)^{|S|},
	\end{equation}
	where $\Phi=\cup_{v\in S} \Phi_v$ is considered as a subset of the set of characters of $A$.
	We denote by $\Pi_v^\vee$ the set of fundamental weights corresponsing to $\Pi_v$.
	We recall that a weight $\chi\in \Phi_v$ is called dominant if 
	$$
	\chi=\prod_{\psi\in \Pi_v^\vee} \psi^{n_\psi}
	$$
	with some non-negative integers $n_\psi$. Since every $\chi\in \Phi_v$ is of the form
	$$
	\chi=\psi \prod_{\alpha\in\Pi_v} \alpha^{-s_\alpha}
	$$
	for some dominant weight $\psi$ and non-negative integers $s_\alpha$, it follows that for every $a\in A^+$,
	\begin{equation}
	\label{eq:domeq}
	\max_{\chi\in \Phi} |\chi(a)|=\max_{\psi\in \Phi^{dom}} |\psi(a)|,
	\end{equation}
	where $\Phi^{dom}$ denotes the subset of dominant weights of $\Phi$. 
	For every $\psi\in \Pi_v^\vee$, there
	exists $\ell\in \mathbb{N}$ such that
	$$
	\psi^\ell=\prod_{\alpha\in \Pi_v} \alpha^{m_\alpha}
	$$
	for some positive integers $m_\alpha$ (see \cite[Ch. 3, \S1.9]{OV}).
	Here we used that since $\bG$ is absolutely simple, the root systems $\Sigma_v$ are irreducible.
	Hence, we deduce that there exists $c_1''\ge 1$ such that
	for every $a\in A^+$,
	$$
	\left( \max_{\alpha\in \Pi} |\alpha(a)|\right)^{(c_1'')^{-1}}\le \max_{\psi\in \Phi^{dom}} |\psi(a)|\le \left( \max_{\alpha\in \Pi} |\alpha(a)|\right)^{c_1''}.
	$$
	Combining this estimate with \eqref{eq:aa},\eqref{eq:aaa} and \eqref{eq:domeq}, we deduce (iii).
	
	Finally, the claim (iv) is proved exactly as in Lemma \ref{lemma1_outline_main1}.
\end{proof}

\subsection{Reductions in the proof of Theorems \ref{main1_Sarithmetic} and \ref{main1_coupl_Sarithmetic}}

The proof of Theorems \ref{main1_Sarithmetic} and \ref{main1_coupl_Sarithmetic}
follows the same steps as the proof of Theorem \ref{main1}, and we freely use the notation introduced in Section \ref{outline_main1_0}.
It is clear that Theorem \ref{main1_Sarithmetic} 
is a particular case of Theorem \ref{main1_coupl_Sarithmetic} with $\xi=m_{\Delta_{[k]}(X)}$.
As in Section \ref{outline_main1_0}, we introduce the projective tensor product norms. 
For $I\subset [k]$, we denote by $\cC_c^\infty(X)^U_I$ the algebraic tensor product of the 
algebras $\cC_c^\infty(X)^U$ over the set of indices in $I$.
For a function $\phi \in\cC_c^\infty(X)_I^U$, we define
\[
\cS_{d,U,I}(\phi) := \inf\Big\{ \sum_{j} \cS_d(\phi_{i_1\,j}) \cdots \cS_d(\phi_{i_l\,j}) \Big\},
\]
where the infimum is taken over all possible ways to write $\phi$ as a finite sum of the form
\[
\phi = \sum_j \phi_{i_1\,j} \otimes \cdots \otimes \phi_{i_l\,j}, \quad\quad \textrm{with $\phi_{i_1\,j},\ldots,\phi_{i_l\,j} \in\cC_c^\infty(X)^U$}.
\]
Theorem \ref{main1_coupl_Sarithmetic} can be reformulated in terms of the Wasserstein distance 
$\dist_{d,U,I}=\dist_{\cS_{d,U,I}}$
as:\\

\noindent {\it
For all $k \ge 2$ and sufficiently large $d$, there exists $\delta=\delta(k,d) > 0$ such that for all
compact open subgroups $U\subset G_f$ and $g_{[k]} = (g_1,\ldots,g_k) \in G_{[k]}$, 
$$
\dist_{d,U,[k]}(g_{[k]}^{-1} \cdot \xi,m_{[k]}) \ll_{d,U,k} \mathfrak{H}(g_{[k]})^{-\delta}.
$$
}\\

Theorem \ref{main1_coupl_Sarithmetic} will be deduced from the following general
inductive estimate which generalises Theorem \ref{main1_real}.

\begin{theorem}
	\label{main1_padic}
	Fix $d,r\in\bN$ such that \eqref{a1}--\eqref{a4} hold and a compact open subgroup $U$ of $G_f$.
	Fix $q \geq 1$, an integer $k \geq 2$, and a $k$-tuple $g_{[k]} = (g_1,\ldots,g_k) \in G_{[k]}$. Suppose that
	\begin{itemize}
		\item $\xi$ is a $\Delta_{[k]}(G)$-invariant $k$-coupling of $(X,m)$.
		\item There exists $\eps > 0$ such that 
		\begin{equation}
		\label{defQ_1}
		\max_{i \neq j} \|g_i^{-1} g_j\|_{\rm op}\geq q^{\eps}.
		\end{equation}
		
		\item There exists $\delta_2>0$ such that for all $\phi_1,\phi_2\in\cC_c^\infty(X)^U$ and $g\in G$,
		\begin{equation}
		\label{KS_0}
		| m((g \cdot \phi_1) \phi_2) - m(\phi_1) m(\phi_2) | \ll_{d,U,r} \|g\|_{\rm op}^{-\delta_2} \, \cS_{d+r}(\phi_1) \, \cS_{d+r}(\phi_2), 
		\end{equation}

		\item There exist $F \geq 1$ and $\tau > 0$ such that
		\begin{equation}
		\label{k-1ass_1}
		\dist_{d,U,I}(g_I^{-1} \cdot \xi_I, m_I) \leq F\, q^{-\tau}, \quad \textrm{for all $I \subsetneq [k]$}.
		\end{equation}
	\end{itemize}
	Then there exists $\gamma_k > 0$, which only depends on $k$, $d$, $r$ and $\delta_2$, such that 
	\begin{equation}
	\label{conclus_1}
	\dist_{d + r,U,[k]}(g_{[k]}^{-1} \cdot \xi,m_{[k]}) \ll_{d,U,r, k} \sqrt{F} \, q^{-\gamma_k \min(\eps,\tau)}.
	\end{equation}
\end{theorem}

It is straightforward to deduce Theorems \ref{main1_Sarithmetic} and \ref{main1_coupl_Sarithmetic} from 
Theorem \ref{main1_padic} by taking $q=\min_{i\ne j} \|g_ig_j^{-1}\|_{\rm op}$ and $\epsilon=1$ (cf. 
the proof of Theorems \ref{main1} and \ref{main1_coupl} in Section \ref{sec:proof_lie}), so that we omit the details. 
Although the parameter $\epsilon$ is not needed for the proof of 
Theorems \ref{main1_Sarithmetic} and \ref{main1_coupl_Sarithmetic}, it will be important when estimating higher order correlations for adele groups. The rest of this section occupies the proof of Theorem \ref{main1_padic}.

\subsection{Proof of Theorem \ref{main1_padic}}\label{sec:hij}

We proceed as in Section \ref{sec:p2}.
We set 
$$
Q: = \max_{i\ne j}\|g_{i}^{-1} g_{j}\|_{\rm op}
$$
and fix indices $i_0\ne j_0 \in [k]$ such that $Q = \|g_{i_0}^{-1}g_{j_0} \|_{\rm op}$.
By Lemma \ref{lemma1_outline_main1_padic}(ii), 
there exists $Z \in \gog_v$  for some $v\in S$ such that 
$\Ad(Z)$
is a nilpotent endomorphism of $\mathfrak{g}$, $\|Z\| = 1$, and
\[
Q = \|g_{i_0}^{-1} g_{j_0}\|_{\rm op} \le c_0 \|\Ad(g_{i_0}^{-1} g_{j_0})Z\|.
\]
After reindexing, we may assume that
\[
\|\Ad(g_{1}^{-1}g_{s})Z\| \geq \|\Ad(g_{2}^{-1}g_{s})Z\| \geq \ldots \geq \|\Ad(g_{k}^{-1}g_{s})Z\|,
\]
and 
$$
\|\Ad(g_{1}^{-1}g_{s})Z\|\ge c_0^{-1}\, Q.
$$
We set 
\[
Z_j = \frac{\Ad(g_{j}^{-1} g_{s})Z}{\|\Ad(g_{1}^{-1} g_{s})Z\|}\quad \hbox{and}
\quad w_j = \|Z_j\|, \quad\quad \textrm{for $j = 1,\ldots,k$}.
\]
Then
\begin{equation}
\label{normalization_padic}
1 = w_1 \geq w_2 \geq \ldots \geq w_k\quad \hbox{and} \quad w_k \le \|\Ad(g_{1}^{-1} g_{s})Z\|^{-1}\le c_0\,Q^{-1}.
\end{equation}
We note that all elements $Z_j$ are contained in $\gog_v$ for a fixed  $v\in S$.
We set $\mathbb{K}$ to be either $\mathbb{R}$ or $\mathbb{Q}_p$, so that $\mathbb{K}\subset F_v$, and consider the flows
$h_j : \mathbb{K} \times X \ra X$ defined by
$$
h_j(t) \cdot x = \exp(tZ_j) \cdot x, \quad\quad \textrm{for $x \in X$}.
$$
We fix an index $1 \leq p \le k-1$ and 
consider the decomposition $[k] = I \sqcup J$, where
$I = [1,p]$ and $J = [p+1,k]$.
Then we also have the diagonal flows 
$$
h_I : \bK \times X_I \ra X_I,\quad h_J : \bK \times X_J \ra X_J,\quad 
h : \bK \times X_{[k]} \ra X_{[k]}.
$$
We note that the vectors $Z_j$ are chosen so that the coupling $\eta=g_{[k]}^{-1}\cdot \xi$ is invariant under the flows $h$, and its marginals $\eta_I$
and $\eta_J$ are invariant under the flows $h_I$ and $h_J$ respectively.

We fix a compact open subgroup $U$ of $G_f$, and set $\cA^U=\cC^\infty_c(X)^U$.
For $v\in S_\infty$, we denote by $\rho_{G_v}$ the left-invariant Riemannian metric on $G_v$
defined as in Section \ref{Subsec:notation}. 
For $v\in S_f$, we denote by $\|\cdot \|$ the operator norms on $\hbox{End}(\gog_v)$.
Let $d$ and $r$ be integers so that so that the Sobolev norms $M := \cS_{d}$ and $N := \cS_{d+r}$ on $\cA^U$ satisfy
\begin{align*}
\|\phi\|_\infty &\ll_{d,U} M(\phi)\ll_{d,U,r} N(\phi),  \\
\|\phi-g\cdot \phi\|_\infty &\ll_{d,U,r} \rho_{G_v}(g,e_{G_v}) \, N(\phi)\quad \hbox{for all $g \in G_v$ with $v\in S_\infty$,} \\
\|\phi-g\cdot \phi\|_\infty &\ll_{d,U,r} \|\Ad(g)-id\| \, N(\phi)\quad \hbox{for all $g \in G_v$ with $v\in S_f$,} \\
N(g \cdot \phi) &\ll_{d,U,r} \|g\|_{\rm op}^{\sigma}\, N(\phi)\quad\hbox{for all $g\in G$ and some $\sigma =\sigma(d,r)> 0$},\\
M(\phi_1 \cdot \phi_2) &\ll_{d,U,r} N(\phi_1) N(\phi_2).
\end{align*}
Using these estimates, we establish the following properties of the flows $h_I$ and $h_J$,
which verify the assumptions of Proposition \ref{mainest} for these flows:
\begin{enumerate}
	\item[1.] 
	There exist $A \geq 1$, depending only on $d$, $U$, $r$ and $k$, and $a > 0$,
	depending only on $d$, $r$ and $k$, such that for all $t \in \mathbb{K}$ and $\phi_I \in \cA_I^U$,
	$$
	N_I(h_I(t) \cdot \phi_I) \leq A \max(1,|t|)^a \, N_I(\phi_I).
	$$

	\item[2.] There exists $B \geq 1$, depending only on $d$, $U$, $r$  and $k$, such that for all $t\in \mathbb{K}$ and $\phi_I \in \cA_I^U$,
	$$
	|m_I((h_I(t) \cdot \phi_I) \phi_I) - m_I(\phi_I)^2| \leq B \max(1,w_p |t|)^{-\delta_2} \, N_I(\phi_I)^2.
	$$

	\item[3.] 
	There exists $C \geq 1$, depending only on $d$, $U$, $r$ and $k$, such that for all $t\in\mathbb{K}$ satisfying $|t|\le w_{p+1}^{-1}$ and $\phi_J \in \cA_J^U$,
	$$
	\|h_J(t) \cdot \phi_J - \phi_J\|_\infty \leq C\, w_{p+1} |t|  \, N_J(\phi_J).
	$$
\end{enumerate}

The proof of Properties 1--3 is essentially the same as the proof
 Lemmas \ref{lemma1}, \ref{lemma2}, and \ref{lemma3}, so that 
 we omit the details. We only comment on the proof of the last property when $\bK$ is non-Archemedian. In this case we can argue as in  the proof of Lemma \ref{lemma3},
 and it remains to estimate
 $$
 \|\phi-\exp(tZ_l)\cdot \phi\|_\infty \ll_{d,U,r} \|\Ad(\exp(tZ_l))-id\| \, N(\phi)
 $$
 for $p+1\le l\le k$. Since 
 $$
 \Ad(\exp(tZ_l))=\exp(t\Ad(Z_l))= \sum_{i=0}^{\dim(\gog_v)} \frac{(t\Ad(Z_l))^i}{i!},
 $$
 and 
 $$
 \|t\Ad(Z_l)\|\ll \|t Z_l\|_v \le |t|\,\|Z_{p+1}\|_v=|t| w_{p+1}\le 1,
 $$
 it follows that
 $$
 \|\exp(tZ_l)-id\|\ll w_{p+1} |t|.
 $$
 This implies Property 3.
 \\
 
Next, since the Properties 1--3 hold, we can apply Proposition \ref{mainest} (proved in Section \ref{PrfABCDT}) to deduce that
$$
	\dist_{N_{[k]}}(\eta,m_{[k]}) \ll_{d,U,r,k} \max((\mathcal{M}T^a)^{1/2},(w_p T)^{-\delta_2/2},w_{p+1} T),
$$
	for all $T \in [w_p^{-1}, w_{p+1}^{-1}]$, where
	\[
	\mathcal{M} := \max(\dist_{M_I}(\eta_I,m_I),\dist_{M_J}(\eta_J,m_J)).
	\]
	
In order to finish the proof of Theorem \ref{main1_Sarithmetic},
we need to solve first the following problem: given $Q \geq 1$, we wish to ``minimize'' (over $p$ and $T$ for
which $T\in [w_p^{-1},w_{p+1}^{-1}]$) the 
expression
\[
\max((\mathcal{M}T^a)^{1/2}, (w_p T)^{-\delta_2/2}, w_{p+1} T),
\]
where $w_1,\ldots,w_k$ is a fixed sequence which satisfies
\[
1 = w_1 \geq w_2 \geq \ldots \geq w_k\quad\hbox{and}\quad w_k\le c_0\,Q^{-1}.
\]
We outline below one way to do this, under the assumptions
that $Q$ is not ``too small'' while
$\mathcal{M}$ is ``small''.
To make the notions ``large'' and ``small'' more precise, we fix a parameter $q \geq 1$, and constants $F \geq 1$ and $\eps, \tau > 0$ 
such that
\[
Q \geq q^{\eps} \qand \mathcal{M} \leq F\, q^{-\tau}
\]
(cf. \eqref{defQ_1} and \eqref{k-1ass_1}).
The problem now takes the following form. We wish to bound from above (for some appropriate choices of 
$p$ and $T$ such that $T\in [w_p^{-1}, w_{p+1}^{-1}]$) the expression
\begin{equation}
	\label{maxexp_1}
	\sqrt{F} \max(q^{-\tau/2} T^{a/2}, (w_p T)^{-\delta_2/2}, w_{p+1} T),
\end{equation}
where $w_1,\ldots,w_k$ is a sequence which satisfies
\[
1 = w_1 \geq w_2 \geq \ldots \geq w_k\quad \hbox{and}\quad w_k\leq c_0\,q^{-\eps}.
\]\\

Let us consider a collection points $\theta^i$, $0\le i\le k-1$,
with $\theta\in [w_k^{1/k},1]$.
Since all of these points lie between $w_k$ and $w_1$,
by the Pigeonhole Principle, there exist two consecutive points 
$\theta^{i+1}$ and $\theta^i$ for some $i=0,\ldots, k-2$ that will end up in one and the same 
of the $(k-1)$ intervals
\[
[w_k,w_{k-1}), \ldots, [w_3,w_2), [w_2,w_1].
\]
We fix an index $1 \leq p < k$ for which
\begin{equation}
\label{eq:b0}
w_{p+1} \leq \theta^{i+1}<\theta^i \leq w_p.
\end{equation}
Let $T = \theta^{-i-1/2}$. Then 
\[
w_p T = w_p \theta^{-i-1/2}  \geq \theta^i  \theta^{-i-1/2}  = \theta^{-1/2} \geq 1,
\]
and 
\[
w_{p+1} T \leq \theta^{i+1}  \theta^{-i-1/2}  = \theta^{1/2} \leq 1,
\]
Using \eqref{eq:b0}, we deduce that  \eqref{maxexp_1} is bounded from above by
\begin{equation}
\label{eq:F}
\sqrt{F} \max(q^{-\tau/2} \theta^{-a(k-3/2)/2}, \theta^{\delta_2/4}, \theta^{1/2}).
\end{equation}
We take
$$
\theta=\max\left(w_k^{1/k}, q^{-\tau/a(2k-3)}\right).
$$
Since $w_k \le c_0\,q^{-\eps}$, we deduce that \eqref{eq:F}
satisfies
$$
\ll \sqrt{F} q^{-\gamma_k \min(\eps,\tau) },
$$
where $\gamma_k$ depends only on $k, a$ and $\delta_2$.
More precisely, $\gamma_k\gg 1/k$.\\

 We have shown that if $q \geq 1$ is fixed, $F\ge 1$, and $\tau, \eps > 0$ are constants such that
\[
Q \geq q^\eps  \qand \max(\dist_{M_I}(\eta_I,m_I),\dist_{M_J}(\eta_J,m_J)) \leq F q^{-\tau},
\]
then 
$$
\dist_{N_{[k]}}(\eta,m_{[k]}) \ll_{d,U,r,k} \sqrt{F}\,q^{-\gamma_k \min(\eps,\tau)}.
$$
 This finishes the proof of Theorem \ref{main1_real} modulo the proof of Proposition \ref{mainest} that we will prove in Section \ref{PrfABCDT}.

\section{Higher-order correlations for adele groups}
\label{outline_main1_adele}

Let $\bG\subset \hbox{GL}_n$ be a simply connected absolute simple algebraic group defined over a number field $F$. We denote by $\mathcal{V}_F$ the set of places of $F$. 
For $v\in \mathcal{V}_F$, let $F_v$ be the corresponding completion of $F$. 
For non-Archemedian places $v$, we also denote by $O_v=\{x\in F_v:\, |x|_v\le 1\}$
the ring of integers in $F_v$. Then the adele group
$$
\bG(\bA_F):={\prod_{v\in\mathcal{V}_F}}^{\!\!\!\prime}\, \bG(F_v)
$$
is the restricted direct product with respect to the family of compact open subgroups $\bG(O_v)$. We set
$$
G_\infty:={\prod_{v\in\mathcal{V}^\infty_F}} \bG(F_v)
\quad \hbox{and}\quad 
G_f:={\prod_{v\in\mathcal{V}^f_F}}^{\!\!\!\prime} \,\bG(F_v),
$$
where $\mathcal{V}^\infty_F$ and $\mathcal{V}^f_F$ denote the subsets of Archemedian places
and non-Archemedian places respectively.
We also denote by $U_\infty$ the subgroup of $G_\infty$ consisting of compact factors.
The group  of rational points 
$$
\Gamma:=\bG(F)
$$
embeds in $\bG(\bA_F)$ diagonally as a discrete subgroup with finite covolume.
We will be interested in the action of $\bG(\bA_F)$ on the homogeneous space
$$
X:=\Gamma\backslash \bG(\bA_F)
$$
equipped with the normalised invariant measure $m$.

Given a compact open subgroup $W$ of $G_f$, we denote by $\cC_c^\infty(X)^W$
the algebra of compactly supported functions on $X$ which are smooth with respect to the action of $G_\infty$ and are $W$-invariant. 
Now we introduce a collection of Sobolev norms $\cS_{d,W}$ on $\cC_c^\infty(X)^W$. 
Let us choose a finite collection $S$ of places
which contains all Archemedean places such that the group
$$
G:=\prod_{v\in S} \bG(F_v)
$$
is not compact.
We also set 
$$
D:={\prod_{v\in\mathcal{V}_F\backslash S}}^{\!\!\!\!\!\!\prime}\, \bG(F_v),
$$
so that $\bG(\bA_F)=G\times D$.
We suppose that the compact open subgroup $W$ is of the form 
$U\times V$ where $U$ is a compact open subgroup of $\prod_{v\in S\cap \mathcal{V}^f_F} \bG(F_v)$, and $V$ is a compact open subgroup of $D$.
Let 
$$
\Gamma_V := \Gamma \cap ({G}\times V).
$$
It will be also convenient to consider $\Gamma_V$ as a subgroup of $G$
by identifying it with the corresponding projection.
It follows from the Strong Approximation Theorem \cite[\S7.4]{PR}
that the projection of $\Gamma$ to $D$ is dense.
Using this, one can check  that the map 
\begin{equation}
\label{map}
\Gamma_V \backslash G \ra \Gamma \backslash (G \times D)/V: \; \Gamma_V g \mapsto \Gamma(g,e_D)V
\end{equation}
is a $G$-equivariant homeomorphism. 
In particular, $\Gamma_V$ is a lattice in $G$.
We set
\[
X_V := \Gamma_V \backslash G,
\]
and denote by $m_V$ the invariant probability measure on $X_V$.
Using that \eqref{map} is a homeomorphism,
we see that the map $\cC_c(X_V)\ra\cC_c(X)^V$ given by $\phi \mapsto F_\phi$, where
\begin{equation}
\label{mapf}
F_\phi(\Gamma(g,e_D))=\phi(\Gamma_V g), \quad\quad \textrm{for $\Gamma_V g \in X_V$},
\end{equation}
is a well-defined isomorphism, and
\[
\int_X F_\phi \, dm = \int_{X_V} \phi \, dm_V, \quad \quad \textrm{for all $\phi \in\cC_c(X_V)$}.
\]
This map also induces the isomorphism $\cC^\infty_c(X)^W \cong \cC^\infty_c(X_V)^U$.
Using this identification, we introduce Sobolev norms on $\cC^\infty_c(X)^U$.
For an integer $d$, we define the \emph{Sobolev norm} $\cS_{d,W}$ on $\cC^\infty_c(X)^W$ of 
order $d$ and level $W$, by
\begin{equation}\label{eq:SSS}
\cS_{d,W}(F_\phi) := \cS_d(\phi), \quad \textrm{for $\phi \in \cC^\infty_c(X_V)^U$},
\end{equation}
where $\cS_d$ is the Sobolev norm on the $S$-algebraic homogeneous space
as in Section \ref{sec:salg}.


\subsection{A reformulation of Theorem \ref{main2}}

Let $S$ be the subset of $\mathcal{V}_F^\infty$ 
consisting of $v$ such that $\bG(F_v)$ is not compact and $R=\mathcal{V}_F^\infty\backslash S$.
According to our assumption on $\bG$, $S\ne \emptyset$.
We set 
$$
G:={\prod_{v\in S}} \bG(F_v)\quad\hbox{and}\quad
U_\infty :={\prod_{v\in R}} \bG(F_v).
$$
When $R=\emptyset$, then $G_\infty$ has no compact factors, and we set $U_\infty=1$.
We observe that for a compact open subgroup $W$ of $G_f$, we have 
$$
C_c^\infty(X)^{U_\infty W}\cong \cC^\infty_c(X/U_\infty)^W\quad
\hbox{and}\quad X/U_\infty\cong \Gamma\backslash
\bG(\bA_F^R),
$$
where
$$
\bG(\bA_F^R)=:
{\prod_{v\in\mathcal{V}_F\backslash R}}^{\!\!\!\!\!\!\prime}\, \bG(F_v),
$$
and $\Gamma$ is identified with its projection to $\bG(\bA_F^R)$.
Hence, it sufficient to prove Theorem \ref{main2} for functions 
$\phi_1,\ldots,\phi_k\in \cC^\infty_c(\Gamma\backslash \bG(\bA_F^R))^W$
and $s_1,\ldots,s_k\in \bG(\bA_F^R)$.
We also set
$$
D:={\prod_{v\in\mathcal{V}_F\backslash \mathcal{V}^\infty_F}}^{\!\!\!\!\!\!\!\!\!\prime}\, \bG(F_v),
$$
so that $\bG(\bA^R_F)=G\times D$.
We note that $\Gamma=\bG(F)$ is embedded  diagonally in $G\times D$
as a lattice. 
It follows from the Strong and Weak Approximation Theorems \cite[Ch.~7]{PR},
the projections of $\Gamma$ to $G$ and $D$, as well as each of the simple factors $G_v$, $v\in S$,
of $G$ are dense. From now on we set
$$
X:=\Gamma\backslash (G\times D).
$$
With these notations, we still have the identifications \eqref{map} and \eqref{mapf}
with $V=W$.

We retain the notation introduced in Sections \ref{outline_main1_0} and \ref{sec:salg}, and take 
\begin{equation}
\label{eq:gd}
g_{[k]} = (g_1,\ldots,g_k) \in G_{[k]} \qand d_{[k]} = (d_1,\ldots,d_k) \in D_{[k]}.
\end{equation}
Since the projection of $\Gamma$ to $D$ is dense in $D$, 
we can find $\gamma_i \in \Gamma$ 
such that 
\begin{equation}
\label{eq:ggg}
d_i \in  \gamma_i W \quad\hbox{ for every $i = 1,\ldots,k$.}
\end{equation}
We set
$\gamma_{[k]} = (\gamma_1,\ldots,\gamma_k) \in \Gamma^k.$
Given ${\phi}_1,\ldots,{\phi}_k \in\cC_c(X_W)$, we consider the corresponding
functions $F_1 =F_{{\phi}_1},\ldots,F_k=F_{{\phi}_k} \in\cC_c(X)^W$
defined via the isomorphism \eqref{mapf} with $V=W$.
We obtain
\begin{align*}
&\big((g_{[k]},d_{[k]}) ^{-1}\cdot m_{\Delta_{[k]}(X)}\big)({F}_1 \otimes \cdots \otimes {F}_k) \\
= &
\int_X F_1(\Gamma(gg_1,dd_1)) \cdots F_k(\Gamma(gg_k,dd_k))
\,
dm(\Gamma(g,d)) \\
= &
\int_X F_1(\Gamma(gg_1,d\gamma_1)) \cdots F_k(\Gamma(gg_k,d\gamma_k))
\,
dm(\Gamma(g,d)) \\
=&
\int_X \left( \int_W F_1(\Gamma(gg_1,dw\gamma_1)) \cdots F_k(\Gamma(gg_k,dw \gamma_k)) \, d\nu_W(w) \right)
dm(\Gamma(g,d)),
\end{align*}
where $\nu_W$ denote the normalised invariant measure on the compact subgroup $W$.
If we define
\[
F(\Gamma(g,d)) := \int_W F_1(\Gamma(gg_1,dw\gamma_1)) \cdots F_k(\Gamma(gg_k,dw \gamma_k)) \, d\nu_W(w),
\]
then clearly $F$ belongs to $\cC_c(X)^W$, and thus
\[
\int_X F \, dm = \int_{X_W} F(\Gamma(g,e_D)) \, dm_W(\Gamma_W g).
\]
On the other hand, since the integrand in the definition of $F$, viewed as a function on the group $W$, is invariant under the open in $D$ subgroup 
$$
W' := \bigcap_{i} \gamma_iW\gamma_i^{-1},
$$
we see that
\begin{eqnarray*}
F(\Gamma(g,e_D)) &=& \frac{1}{|W/W'|} \sum_{w \in W/W'} 
F_1(\Gamma(gg_1,w\gamma_1)) \cdots F_k(\Gamma(gg_k,w \gamma_k)).
\end{eqnarray*}
Since the projection of $\Gamma$ to $D$ is dense, we have $W/W'=\Gamma_{W}/\Gamma_{W'}$, so that
\begin{eqnarray*}
	F(\Gamma(g,e_D)) &=& \frac{1}{|\Gamma_{W}/\Gamma_{W'}|} \sum_{\delta\in \Gamma_{W}/ \Gamma_{W'}} 
	F_1(\Gamma(gg_1,\delta\gamma_1)) \cdots F_k(\Gamma(gg_k,\delta \gamma_k)) 
	\\
	&=& \frac{1}{|\Gamma_{W}/\Gamma_{W'}|} \sum_{\delta\in \Gamma_{W}/ \Gamma_{W'}} 
	F_1(\Gamma(\gamma_1^{-1} \delta^{-1} gg_1,e_D)) \cdots F_k(\Gamma(\gamma_k^{-1} \delta^{-1} gg_k,e_D))
	\\
	&=&
	\frac{1}{|\Gamma_{W}/\Gamma_{W'}|} \sum_{\delta\in \Gamma_{W}/ \Gamma_{W'}} 
	\phi_1(\Gamma_W \gamma_1^{-1} \delta^{-1} gg_1) \cdots \phi_k(\Gamma_W \gamma_k^{-1} \delta^{-1} gg_k).
\end{eqnarray*}
Then 
\begin{eqnarray*}
\int_X F \, dm
&=& 
\int_{X_W} \left( \frac{1}{|\Gamma_{W}/ \Gamma_{W'}|} \sum_{\delta \in \Gamma_{W}/\Gamma_{W'}} 
\phi_1(\Gamma_W \gamma_1^{-1} \delta^{-1} gg_1) \cdots \phi_k(\Gamma_W \gamma_k^{-1} \delta^{-1} gg_k)\right) \, dm_W(\Gamma_W g) \\
&=& 
\int_{X_{W'}}  
\phi_1(\Gamma_W \gamma_1^{-1} gg_1) \cdots \phi_k(\Gamma_W \gamma_k^{-1} gg_k) \, dm_{W'}(\Gamma_W g) \\
&=&
(g_{[k]}^{-1} \cdot \xi_{\gamma_{[k]}})(\phi_1 \otimes \cdots \otimes\phi_k),
\end{eqnarray*}
where $\xi_{\gamma_{[k]}}$ denotes the invariant probability measure 
supported on the  closed $\Delta_{[k]}(G)$-orbit 
\[
\Gamma_{W}^k \gamma_{[k]}^{-1} \Delta_{[k]}(G) \cong \Gamma_{W'} \backslash G
\] 
in $(X_W)_{[k]}$. Hence, we conclude that for all ${\phi}_1,\ldots,{\phi}_k \in\cC_c(X_W)$,
$$
\big((g_{[k]},d_{[k]})^{-1} \cdot m_{\Delta_{[k]}(X)}\big)(F_{{\phi}_1} \otimes \cdots \otimes F_{{\phi}_k}) 
=(g_{[k]}^{-1} \cdot \xi_{\gamma_{[k]}})(\phi_1 \otimes \cdots \otimes\phi_k).
$$
Using the norms $\cS_{d,W}$ defined by \eqref{eq:SSS} with $V=W$ on the algebra $\cC_c^\infty(X)^{W}$,
we introduce the Wasserstein distance $\dist_{d,W,{[k]}}$
on $\mathcal{P}(X_{[k]})$ as in \eqref{defW}.
Then the proof of Theorem \ref{main2} reduces to estimating the distance
$\dist_{d,W,{[k]}}((g_{[k]},d_{[k]})^{-1} \cdot m_{\Delta_{[k]}(X)},m_{[k]})$.
We also introduce the Wasserstein distance $\dist_{d,{[k]}}$
on $\mathcal{P}((X_W)_{[k]})$ defined by the Sobolev norms $\cS_d$ on 
the algebra $\cC_c^\infty(X_W)$.
Then for every positive integer $d$, 
\begin{equation}
\label{niceeq}
\dist_{d,W,{[k]}}((g_{[k]},d_{[k]})^{-1} \cdot m_{\Delta_{[k]}(X)},m_{[k]})
=
\dist_{d,[k]}(g_{[k]}^{-1} \cdot \xi_{\gamma_{[k]}},(m_W)_{[k]}),
\end{equation}
for all $(g_{[k]},d_{[k]}) \in G_{[k]} \times D_{[k]}$,
where $\gamma_{[k]}$ is determined by \eqref{eq:ggg}.
We note that $\xi_{\gamma_{[k]}}$ is obviously a $\Delta_{[k]}(G)$-invariant $k$-coupling of $(X_W,m_W)$, so that we can analyse $\xi_{\gamma_{[k]}}$ using the method of Section \ref{sec:salg}. \\

We recall that the height function $\hbox{H}:\bG(\bA_F)\to\bR^+$ is defined in \eqref{eq:HHH}
in terms of the norms $\|\cdot \|_v$ on $\hbox{M}_n(F_v)$.
We note that $\|\cdot\|_v$ is
invariant under $\bG(O_v)$ for almost all $v$.
Given $g_{[k]}$ and $d_{[k]}$ as in \eqref{eq:gd}, we set
\begin{equation}
\label{defQqH}
Q := \max_{i \neq j} \|g_i^{-1} g_j\|_{\rm op}, \quad\quad q_G := \min_{i \neq j} \|g_i^{-1} g_j\|_{\rm op},\quad \quad q_D := \min_{i\ne j}\hbox{H}(d_i^{-1} d_j),
\end{equation}
where $\|\cdot\|_{\rm op}$ is is the sub-multiplicative function on $G$ defined in \eqref{eq:norm00}.
We also set
\begin{equation}
\label{defq2}
q := \min_{i \neq j} \max(\|g_i^{-1} g_j\|_{\rm op}, \hbox{H}(d_i^{-1} d_j)).
\end{equation}
We wish to show that for every large enough integer $d$, there 
exists $\delta=\delta(k,d) > 0$ such that 
\[
\dist_{d,[k]}(g_{[k]}^{-1} \cdot \xi_{\gamma_{[k]}},(m_W)_{[k]}) 
\ll_{d,W,k} q^{-\delta}.
\]
The proof
will separate between the cases when $Q$ is ``large'' in comparison to $q$ and when $Q$ is ``small''
in comparison to $q$.
To make all of this precise, let us fix $\eps > 0$,
and consider the cases when 
$$
Q\ge q^\eps\quad\hbox{ and }\quad Q<q^\eps.
$$

\subsection{Case I: $Q \geq q^{\eps}$}
We shall apply Theorem \ref{main1_padic} to the $\Delta_{[k]}(G)$-invariant $k$-coupling $\xi = \xi_{\gamma_{[k]}}$
of $(X_W,m_W)$. We note that when $I \subset [k]$ is a singleton, 
the assumption \eqref{k-1ass_1} of Theorem \ref{main1_padic}
is clearly satisfied.
Assume now that we have shown that for sufficiently large $d$, there exist $\delta_{k-1} > 0$ such that
\[
\dist_{d,[k]}(g_I^{-1} \cdot \xi, (m_W)_I) \ll_{d,k} q^{-\delta_{k-1}}, \quad \textrm{for all $I \subsetneq [k]$}.
\]
Since $Q \geq q^\eps$, we conclude applying Theorem \ref{main1_padic} inductively that there exists $\gamma_k > 0$ such that for sufficiently large $d$,
\begin{equation}
\label{step2}
\dist_{d,[k]}(g_{[k]}^{-1} \cdot \xi,(m_W)_{[k]}) \ll_{d,k} q^{-\gamma_k \min(\delta_{k-1},\eps)}.
\end{equation}

\subsection{Case II: $Q < q^\eps$}
Let us now deal with the trickier case when $Q$ is ``small'' in comparison to $q$. A straightforward application
of the property \eqref{a3} for the Sobolev norm $\cS_d$ and its projective tensor products
(cf. Lemma \ref{lemma:tensorprodpreceq}) shows that 
there exists $\sigma =\sigma(d,k)> 0$ such that for all $g_{[k]} \in G_{[k]}$,
\begin{equation}
\label{step1}
\dist_{d, k}(g_{[k]}^{-1} \cdot \xi,(m_W)_{[k]}) \ll_{d,k} Q^{\sigma} \dist_{d, k}(\xi,(m_W)_{[k]}).
\end{equation}

We shall now show how one can estimate the right-hand side in 
\eqref{step1} by utilising a general result by the second author, Margulis and Venkatesh  \cite[Theorem~1.3]{EMV}
which we apply to the measure
 $\xi$. We recall that $\xi$ denotes the normalized invariant measure supported on the closed $\Delta_{[k]}(G)$-orbit 
\[
Y(\gamma_{[k]}) := \Gamma_W^k \gamma_{[k]}^{-1} \Delta_{[k]}(G) \subset \Gamma_W^k \backslash G_{[k]} = (X_W)_{[k]}.
\]
We note that using the restriction of scalars functor,
we can consider $X_W$ as a homogeneous space 
of a real algebraic group defined over $\bQ$.
Since $\bG$ is simply connected and isotropic over $F_v$ for $v\in S$,
the group $\Delta_{[k]}(G)\cong G=\prod_{v\in S} \bG(F_v)$ is generated by
unipotent one-parameter subgroups.
Also the centraliser of $\Delta_{[k]}(G)$ in $G_{[k]}$ is finite.
Hence, the results of \cite{EMV} are applicable.
We observe that
$$
Y(\gamma_{[k]}) \cong \Gamma_{W}' \backslash G,
$$
where 
\[
\Gamma_{W}'=\bigcap_i \gamma_i \Gamma_W \gamma_i^{-1}.
\]
 If the volume of $X_W$ is normalized to be one, we see that the volume of the orbit $Y(\gamma_{[k]})$ equals 
the index
\[
\left|\Gamma_W :  \Gamma_W \cap \Big( \bigcap_{i \neq j} \gamma_j^{-1} \gamma_i \Gamma_W \gamma_i^{-1} \gamma_j\Big)\right|
\]
for every fixed $j$.
For $\gamma\in \Gamma$, we define
$$
\Omega_W(\gamma):=|\Gamma_W : \Gamma_W\cap\gamma \Gamma_W \gamma^{-1}|,
$$
and for $\gamma_{[k]} = (\gamma_1,\ldots,\gamma_k) \in \Gamma^k$,
$$
\Omega_W(\gamma_{[k]}):=\min_{i \neq j} \Omega_W(\gamma_j^{-1} \gamma_i).
$$
Then 
$$
\hbox{vol}\left(Y(\gamma_{[k]})\right)\ge \Omega_W(\gamma_{[k]}).
$$

In order to apply \cite[Theorem~1.3]{EMV}, we need to describe orbits $\Gamma_W^k \gamma_{[k]}^{-1} L$
in $X_W^k$ that support a finite invariant measure, where $L$ is a closed subgroup of $G^k$
such that $\Delta_{[k]}(G)\subset L\subset G^k$.
For a partition $\mathcal{P}$ of $[k]$, we set
$$
\Delta_{\mathcal{P}}(G):=\prod_{I\in \mathcal{P}} \Delta_I(G).
$$
Then by Lemma \ref{l:subgroup} proved below, every such orbit is of the form
\begin{equation}\label{eq:orb}
\Gamma_W^k \gamma_{[k]}^{-1} L=\Gamma_W^k \gamma_{[k]}^{-1} \Delta_{\mathcal{P}}(G)Z
\end{equation}
for some partition $\mathcal{P}$ of $[k]$ and a finite subgroup $Z$ of $Z(G)^k$.
We note that then the orbit 
$$
\Gamma_W^{k}\gamma_{[k]}^{-1}\Delta_{\mathcal{P}}(G)Z \subset X_{[k]}
$$ is 
again closed. We observe that
$$
\Gamma_W^{k}\gamma_{[k]}^{-1}\Delta_{\mathcal{P}}(G) \cong \Gamma'_{W,\cP} \backslash G^\cP,
$$ 
where 
\[
\Gamma'_{W,\cP} 
=\prod_{I\in \mathcal{P}} \left(\bigcap_{i\in I} \gamma_i \Gamma_W \gamma_i^{-1}\right).
\]
Hence, if the partition $\mathcal{P}$ is proper,
\begin{equation}\label{eq:v}
\Vol(\Gamma_W^{k}\gamma_{[k]}^{-1}\Delta_{\mathcal{P}}(G)Z) \gg \min_{i \neq j} |\Gamma_W :  \Gamma_W \cap \gamma_j^{-1} \gamma_i \Gamma_W \gamma_i^{-1} \gamma_j| = \Omega_W(\gamma_{[k]}).
\end{equation}
Now we apply \cite[Theorem~1.3]{EMV}. Let us assume that the parameter $\Omega_W(\gamma_{[k]})$ is sufficiently large, so that
it follows from \eqref{eq:v} that $(X_W)_{[k]}=\Gamma_W^{k}\backslash G_{[k]}$
is the only orbit with volume less than $\Omega_W(\gamma_{[k]})^{1/2}$.
Then, by \cite[Theorem~1.3]{EMV}, there exists $\delta_D > 0$ such that
for sufficiently large integers $d$,
\begin{equation}\label{eq:d1}
\dist_{d,[k]}(\xi,(m_W)_{[k]}) \le \hbox{vol}\left(Y(\gamma_{[k]})\right)^{-\delta_D/2} \le  \Omega_W(\gamma_{[k]})^{-\delta_D/2}.
\end{equation}
Although this bound holds only when $\Omega_W(\gamma_{[k]})$
is sufficiently large, since $\dist_{d,[k]}\ll_{d,k} 1$ (cf. Lemma \ref{lemma:tensorprodpreceq}),
we also have 
\[
\dist_{d,[k]}(\xi,(m_W)_{[k]}) \ll_{d,k}  \Omega_W(\gamma_{[k]})^{-\delta_D/2}
\]
in general. By Lemma \ref{l:volume} proved below,
\begin{equation}
\label{eq:omega}
\Omega_W(\gamma_{[k]})\gg_W \min_{i\ne j}\hbox{H}(\gamma_i^{-1}\gamma_j)^\theta .
\end{equation}
We recall that $\hbox{H}$ is defined as the product of the norms $\|\cdot\|_v$
which are bi-$\bG(O_v)$-invariant for almost all $v$, so that
for all $w_1,w_2\in W$ and $d\in D$,
$$
\hbox{H}(w_1dw_2)\gg_W \hbox{H}(d).
$$
Hence, it follows from \eqref{eq:ggg} that
$$
\min_{i\ne j}\hbox{H}(\gamma_i^{-1}\gamma_j) \gg_W  \min_{i\ne j} \hbox{H}(d_i^{-1}d_j) =q_D.
$$
Hence, we deduce from \eqref{eq:d1} that 
\begin{equation}
\label{effcongmix1}
\dist_{d,[k]}(\xi,(m_W)_{[k]}) \ll_{d,W,k} q_D^{-\delta_D \theta/2}.
\end{equation}

\subsection{Combining the two cases}
Now combine the estimates \eqref{step2}, \eqref{step1}, and \eqref{effcongmix1} to complete the proof of Theorem \ref{main2}.
We stress that $\eps > 0$ so far has been a free variable. However, we note that as long as $\eps < 1$, 
then the inequality $Q < q^\eps$ (Case II) implies that $q = q_D$. Indeed, if 
\[
Q < q^\eps \qand q > q_D,
\] 
then the latter inequality readily implies that there exists at least one pair $(i,j)$ of indices with $i \neq j$ such that 
$\|g_i^{-1} g_j\|_{\rm op} > \hbox{H}(d_i^{-1} d_j)$, and thus $Q \geq q$, which contradicts the first inequality if $\eps < 1$.
Hence, as long as $\eps < 1$ and $Q < q^\eps$, \eqref{step1} and \eqref{effcongmix1} together imply that
\[
\dist_{d,[k]}(g_{[k]}^{-1} \cdot \xi,(m_W)_{[k]}) \ll_{d,W,k} q^{\eps\sigma} q^{-\delta_D \theta/2} = q^{-(\delta_D \theta/2 - \eps \sigma)}.
\]
On the other hand, if $Q \geq q^\eps$, then \eqref{step2} asserts that
\[
\dist_{d,[k]}(g_{[k]}^{-1} \cdot \xi,(m_W)_{[k]}) \ll_{d,k} q^{-\gamma_k \min(\delta_{k-1},\eps)}.
\]
Let us now choose 
\[
\eps = \min\left\{\frac{1}{2},\frac{\delta_D\theta}{4 \sigma}\right\},
\]
so that if $Q < q^\eps$, then
\[
\dist_{d,[k]}(g_{[k]}^{-1} \cdot \xi,(m_W)_{[k]}) \ll_{d,W,k}q^{-\delta_D \theta/4},
\]
and if $Q \geq q^\eps$, then
\[
\dist_{d,[k]}(g_{[k]}^{-1} \cdot \xi,(m_W)_{[k]}) \ll_{d,k} q^{-\gamma_k \min(\delta_{k-1},\frac{\delta_D\theta}{4 \sigma})}.
\]
If we denote by $\delta_k$ the minimum of the two exponents above, then 
\[
\dist_{d,[k]}(g_{[k]}^{-1} \cdot \xi,(m_W)_{[k]}) \ll_{d,W,k} q^{-\delta_{k}}.
\]
Finally, we observe that by Lemma \ref{lemma1_outline_main1_padic}(iii), there exists $c\in (0,1)$ such that
$$
\|g_i^{-1}g_j\|_{\rm op}\gg \hbox{H}(g_i^{-1} g_j)^c
$$
for all $i,j$.
Since $\hbox{H}\gg 1$ on $G$ and on $D$, and $\hbox{H}((g,d))=\hbox{H}(g)\hbox{H}(d)$
for all $g\in G$ and $d\in D$,
we deduce that
$$
\max(\|g_i^{-1} g_j\|_{\rm op}, \hbox{H}(d_i^{-1} d_j))\gg 
\max(\hbox{H}(g_i^{-1} g_j)^c, \hbox{H}(d_i^{-1} d_j))
\gg \hbox{H}((g_i,d_i)^{-1}(g_j,d_j))^{c/2}.
$$
Hence, we obtain
\[
\dist_{d,[k]}(g_{[k]}^{-1} \cdot \xi,(m_W)_{[k]}) \ll_{d,W,k} \mathfrak{H}((g_{[k]},d_{[k]}))^{-\delta_{k}c/2}.
\]
Because of \eqref{niceeq}, this finishes the proof of Theorem \ref{main2}, modulo Lemmas \ref{l:subgroup} and \ref{l:volume}.

\subsection{Intermediate subgroups}

We prove the description of the intermediate orbits stated in \eqref{eq:orb}.

\begin{lemma}
	\label{l:subgroup}
	Let $\Gamma_1,\ldots,\Gamma_k$ be irreducible lattices in $G$ and
	$X_{[k]}=X_1\times \cdots \times X_k$ where $X_i=\Gamma_i\backslash G$.
	Suppose that $L$ is an immersed subgroup of $G^k$ containing 
	the diagonal $\Delta_{[k]}(G)$
	such that for some $x_{[k]}=(x_1,\ldots,x_k)\in X_{[k]}$, 
	the orbit 
	$x_{[k]} L$ in $X_{[k]}$ supports a finite invariant measure. Then
	$$
	 x_{[k]} L=x_{[k]}\Delta_{\mathcal{P}}(G)Z
	$$
	for some partition 	$\mathcal{P}$ of $[k]$ and a finite subgroup $Z$ of the centre $Z(G)^{k}$.
\end{lemma}

We note that since $\bG$ is simply connected and isotropic over $F_v$ for $v\in S$, by the Strong Approximation Theorem \cite[\S7.4]{PR}, $\Gamma_W$ is an irreducible lattice in $G$, so that this lemma is applicable in our case.

\begin{proof}
	By \cite[\S7.2]{PR},
	every normal subgroup of $G_v=\bG(K_v)$ for $v\in S$,
	is central. We may replace $G$ by $G/Z(G)$ and $\Gamma_i$ by $\Gamma_i Z(G)/Z(G)$
	and carry out the proof when the centre is trivial. 
	To simplify our presentation, we abuse notation and assume that $Z(G)=\{e\}$.
	Using that $G_v$'s are non-commutative and simple,
	it is easy to deduce that every normal subgroup of $G$ is of the form
	$\prod_{v\in S'} G_v$ for some $S'\subset S$.
	Moreover, any normal subgroup $N$ of $G^{k}$  is of the form 
	\begin{equation}
	\label{eq:N}
	N=N_1\times \cdots \times N_k,
	\end{equation}
	where $N_i=\prod_{v\in S_i} G_v$ for some $S_i\subset S$.
	We note that if $N$ is such a subgroup, it follows from irreducibility of lattices $\Gamma_i$ that
	\begin{equation}
	\label{eq:closure}
	\overline{x_{[k]} N} \supset x_{[k]} G_I,
	\end{equation}
	where $I=\{i:\, S_i\ne \emptyset\}$.

	We note that the argument of \cite[Th.~1.13]{rag}
	can be extended to immersed subgroups (namely, to subgroups given by 
	continuous embeddings $L\to G^{k}$), and 
	since the orbit $x_{[k]}L$ supports finite $L$-invariant measure,
	it follows that $x_{[k]}L$ is closed in $X_{[k]}$.
	
	We say that $x_i$ is commensurable with $x_j$ if the subgroups $\hbox{Stab}_G(x_i)$  and $\hbox{Stab}_G(x_j)$ are commensurable.
	Suppose that $L\subset \Delta_\cP(G)$ for some 
	proper partition $\cP$ of $[k]$
	such that for every $I\in \cP$, the points $x_i$, $i\in I$,
	are commensurable. Then 
	$$
	x_{[k]}\Delta_\cP(G)\simeq \prod_{I\in \cP} \Gamma_I\backslash G,
	$$
	where $\Gamma_I=\cap_{i\in I} \hbox{Stab}_G(x_i)$ are irreducible lattices in $G$. Hence, in this case we can reduce the number of factors,
	so that, without loss of generality, we may assume that such partition does not exists.
	
	We claim that under this assumption, $x_{[k]}L=X_{[k]}$  and proceed by 
	induction on $k$. The statement is clear when $k=1$. 
	We consider the decomposition $G^{k}=G^{k-1}\times G$.
	Let $L_1$ and $L_2$ denote the projections of $L$ to
	each of the factors. Since $\Delta_{[k]}(G)\subset L$,  it clear that $L_2=G$.
	Suppose that $(l_1,l_2),(l_1,l_2')\in L$ for some $l_1\in L_1$ and $l_2\ne l'_2\in G$.
	Then $(e,l_2^{-1}l_2')\in L$, and since $\Delta_{[k]}(G)\subset L$,
	we deduce that $\{e\}\times N\subset L$ for some non-trivial normal subgroup $N$	of $G$. As we observed above, $N=\prod_{v\in S'} G_v$ for some non-empty $S'\subset S$.
	Since the orbit $x_{[k]}L$ is closed,
	we deduce from \eqref{eq:closure} that 
	$$
	x_{[k]}L=x_{[k]}L(\{e\}\times G)=x_{[k]}(L_1\times G).
	$$
	Hence, in this case our analysis reduces to understanding finite-volume orbits in the space $X_{[k-1]}$, so that the claim follows from the inductive hypothesis.
	
	Now we suppose that for every $l_1\in L_1$, there exists unique $l_2\in L_2=G$ such that $(l_1,l_2)\in L$.
	Namely, there exists a surjective map $\phi:L_1\to G$ such that
	$L=\{(l,\phi(l)):\, l\in L_1\}$.
	It follows from uniqueness that $\phi$ is a homomorphism,
	and that 
	\begin{equation}
	\label{eq:diag}
	\phi(g,\ldots,g)=g\quad\hbox{ for all $g\in G$.}
	\end{equation}
	We observe that the orbit $x_{[k-1]}L_1$ in $X_{[k-1]}$ supports a finite invariant measure,
	which is the push-forward of the finite invariant measure on $x_{[k]}L$.
	Hence, we can apply the inductive assumption to 
	$x_{[k-1]}L_1$ to deduce that 
	$x_{[k-1]}L_1=x_{[k-1]}G^{k-1}$.
	This implies that the subgroup  $L_1$ is open in $G^{k-1}$.
	Since $\bG$ is simply connected, $G=\prod_{v\in S} \bG(F_v)$ is connected (see \cite[Prop.~7.6]{PR}),
	so that $L_1=G^{k-1}$. We have shown that
	$$
	L=\{(g,\phi(g)): g\in G^{k-1}\},
	$$
	where $\phi:G^{k-1}\to G$ is a surjective homomorphism.
	Let $N$ be the kernel of $\phi$. 
	Using \eqref{eq:diag}, we deduce that $N$ is non-trivial unless $k=2$.
	Moreover, if $N$ is trivial, it follows from \eqref{eq:diag} that $L=\Delta_{[2]}(G)$,
	so that the lemma holds. Hence, we can suppose that $N\ne \{e\}$.
	The subgroup $N$ is normal in $G^{k-1}$,
	so that it is of the form \eqref{eq:N}.
	In particular, it follows that there exists a closed normal subgroup $M$ of $G^{k-1}$
	commuting with $N$
	such that $G^{k-1}=MN$ and $M\cap N=\{e\}$. Hence,
	$$
	L= \{(nm,\phi(m)):\, n\in N,m\in M\}=(N\times\{e\})\{(m,\phi(m)): \,m\in M\}.
	$$
	Let $I$ be the subset of $[k-1]$ such that $N_i\ne \{e\}$ for $i\in I$.
	Then since the orbit $x_{[k]}L$ is closed, and the lattices $\Gamma_i$ are irreducible in $G$,
	it follows that 
	$$
	x_{[k]}L=x_{[k]}L \left(\prod_{i\in I} G\right).
	$$
	Hence, if $I\ne \emptyset$, we can complete the proof by induction.
\end{proof}

\subsection{Volume estimates}

We prove the estimate for $\Omega_W(\gamma)$ which was used in \eqref{eq:omega}.

\begin{lemma}
	\label{l:volume}
	There exists $\theta>0$ such that for every $\gamma\in \Gamma$, 
	$$
	\Omega_W(\gamma)\gg_W \hbox{\rm H}(\gamma)^\theta.
	$$
\end{lemma}

\begin{proof}
We first observe that $\Omega_W(\gamma)$ can be interpreted in terms of volumes
of suitable subsets of $D$. We recall that $D$ is the restricted product of $\bG(F_v)$, $v\in \mathcal{V}^f_F$. We denote $\nu_D$ the invariant measure on $D$ which is the product of 
invariant measures $\nu_v$ on $\bG(F_v)$ such that $\nu_v(\bG(O_v))=1$ for almost all $v$.
We normalise $\nu_D$ so that $\nu_D(W)=1$. Then
since the projection of $\Gamma$ to $D$ is dense,
\[
\Omega_W(\gamma)= |W : W\cap\gamma W \gamma^{-1}| = \nu_D(W\gamma W).
\]
We claim that there exists $\theta>0$ such that for every $d\in D$
\begin{equation}\label{eq:vvv}
\nu_D(WdW) \gg_W \hbox{H}(d)^\theta.
\end{equation}
This will imply the lemma.
	
For almost all places $v$, the group $K_v=\bG(O_v)$ is a hyperspecial maximal compact open subgroup 
of $\bG(F_v)$ (see \cite{Ti}). For the other places $v\in \mathcal{V}_F^f$, 
we fix a good maximal compact open subgroup $K_v$ of $\bG(F_v)$.
Let $K=\prod_{v\in \mathcal{V}^f_F} K_v$. Then $K$ is a compact open subgroup of $D$,
so that it is commensurable with $W$, and we have 
$$
\nu_D(WdW)\gg_W \nu_D(KdK).
$$
Now it will be convenient to normalise the measures $\nu_v$ on $\bG(F_v)$ so that $\nu_v(K_v)=1$.
We claim that there exists $\theta>0$ such that for every $d_v\in \bG(F_v)$,
\begin{equation}\label{eq:vvv1}
\nu_v(K_vd_vK_v)\gg_v \|d_v\|_v^\theta,
\end{equation}
and moreover for almost all $v$,
\begin{equation}\label{eq:vvv2}
\nu_v(K_vd_vK_v)\ge \|d_v\|_v^\theta.
\end{equation}
Since $\hbox{H}$ is defined as a product of the norms $\|\cdot \|_v$, this will imply \eqref{eq:vvv}.

We recall the Cartan decomposition $\bG(F_v)=K_vA_v^+\Omega_v K_v$ introduced in \eqref{eq:cartan_padic}. For almost all $v$, the group $\bG$ is quasi-split over $F$
and split over unramified extension of $F$. In this case, we have the Cartan decomposition
with $\Omega_v=\{e\}$. 
For $d_v=k_1a_v\omega k_2\in K_vA_v^+\Omega_v K_v$,
$$
\nu_v(K_vd_vK_v)=\nu_v(K_va_v \omega K_v)\ge \nu_v(K_va_vK'_v)\gg_v  \nu_v(K_va_vK_v),
$$
where $K'_v=\cap_{\omega\in\Omega_v\cup\{e\}} \omega K_v \omega^{-1}$
is a compact open subgroup of $K_v$, and 
$$
\|d_v\|_v\ll_v \|a_v\|_v.
$$
Moreover, for almost all $v$,
$$
K_vd_vK_v=K_va_vK_v\quad\hbox{and}\quad \|d_v\|_v = \|a_v\|_v.
$$
Hence, it is sufficient to prove \eqref{eq:vvv1} and \eqref{eq:vvv2} when $d_v=a_v\in A_v^+$.

Let $\Delta_v$ denotes the product of all positive roots of $\bA_v$.
It follows from \cite[3.2.15]{mac} that if $K_v$ is a hyperspecial maximal compact
subgroup of $\bG(F_v)$, then 
\begin{equation}
\label{eq:v1}
\nu_v(K_va_vK_v)\ge |\Delta_v(a_v)|_v,\quad \hbox{for $a_v\in A_v^+$.}
\end{equation}
In particular, this bound holds for almost all places $v$.
For the other places $v$, we also have
\begin{equation}
\label{eq:v2}
\nu_v(K_va_vK_v)\gg_v |\Delta_v(a_v)|_v,\quad \hbox{for $a_v\in A_v^+$.}
\end{equation}
On the other hand, we recall from the proof of Lemma \ref{lemma1_outline_main1_padic}(iii)
that for all $a_v\in A_v^+$,
$$
\|a_v\|_v\ll_v \max_{\psi\in \Phi_v^{dom}} |\psi(a_v)|_v
$$
and there exists $\theta'>0$ such that
$$
\max_{\psi\in \Phi_v^{dom}} |\psi(a_v)|_v \le \left(\max_{\psi\in \Pi_v} |\alpha(a_v)|_v\right)^{\theta'}.
$$
Hence, combining this estimate with \eqref{eq:v2}, we deduce \eqref{eq:vvv1}.
Further, by \cite[Lemma~6.4]{STBT}, which also extends to reducible representations,
we obtain that for almost all places $v$, 
$$
\|a_v\|_v =\max_{\psi\in \Phi_v^{dom}} |\psi(a_v)|_v
$$
for all $a_v\in A_v^+$,
so that \eqref{eq:vvv2} follows from \eqref{eq:v1}. This completes the proof of the lemma.
\end{proof}

\section{Wasserstein distances on couplings}
\label{sec:wasserstein}

\subsection{Wasserstein distances}

Let $X$ be a locally compact Hausdorff space. We denote by $\cC_c(X)$ the space of continuous
functions on $X$ with compact supports, equipped with the topology of uniform convergence on compact subsets, 
and we write $\cP(X)$ for the space of Borel probability measures on $X$, which we shall think of as non-negative elements in the dual space $\cC_c(X)^*$.

Given a linear subspace $\cA \subset\cC_c(X)$ and a norm $M$ on $\cA$, we define the \emph{Wasserstein distance}
$\dist_M$ on $\cP(X)$ by
\[
\dist_M(\mu,\nu)
:= 
\sup\Big\{
| \mu(\phi) - \nu(\phi) | \, : \, \phi \in \cA\;\hbox{ with } M(\phi) \leq 1 \Big\}
\]
for $\mu, \nu \in \cP(X)$. 
We see that this is indeed a distance (metric) if $\cA$ is dense in $\cC_c(X)$, otherwise it is only a semi-distance (semi-metric). The following lemma is immediate. 

\begin{lemma}
\label{lemma:normbnds}
If $M \leq E \, N$ for a constant $E > 0$, then $\dist_{N} \leq E \, \dist_M$.  
\end{lemma}

We say that a norm $M$ on $\cA$ is \emph{uniform} if there exists a constant $F>0$ such that 
$$
\|\phi \|_\infty \leq F\, M(\phi)\quad\quad \hbox{for all $\phi \in \cA$,}
$$  where $\| \cdot \|_\infty$ denotes the uniform norm on 
$\cC_c(X)$. Throughout this paper, all norms that we shall consider will be assumed to be uniform. \\

If the linear subspace $\cA \subset\cC_c(X)$ in addition is closed under multiplication, that is to say, if 
$\cA$ is a subalgebra of $\cC_c(X)$, and $M$ and $N$ are \emph{uniform} norms on $\cA$, then we 
write $M \preceq N$ if there exist constant $D_1, D_2>0$ such that
\[
M(\phi_1) \leq D_1\, N(\phi_1) \qand M(\phi_1\cdot \phi_2) \leq D_2\, N(\phi_1) N(\phi_2), \quad \textrm{for all $\phi_1, \phi_2 \in \cA$}.
\]

\subsection{Projective tensor product norms}

Let us now assume that $X_1$ and $X_2$ are locally compact Hausdorff spaces, and fix subalgebras
\[
\cA_1 \subset\cC_c(X_1) \qand \cA_2 \subset\cC_c(X_2).
\]
Let $\cA_1 \otimes \cA_2$ denote the (algebraic) tensor product of the algebras $\cA_1$ and $\cA_2$, i.e. the subalgebra of $\cC_c(X_1 \times X_2)$ which consists of functions which are finite sums of the form
\[
\sum_{i} (\phi_{1i} \otimes \phi_{2i})(x_1,x_2),\quad \quad \textrm{for $(x_1,x_2) \in X_1 \times X_2$},
\]
where $\phi_{1i}$ and $\phi_{2i}$ are elements of $\cA_1$ and $\cA_2$ respectively. If $M_1$ and $M_2$ are norms on $\cA_1$ and $\cA_2$ respectively, we define the \emph{projective tensor product norm} (or \emph{maximal cross-norm}) $M_1 \otimes M_2$ on the algebra $\cA_1 \otimes \cA_2$
by 
\[
(M_1 \otimes M_2)(\phi) := \inf\left\{ \sum_i M_1(\phi_{1i}) M_2(\phi_{2i}) \, : \, \phi = \sum_i \phi_{1i} \otimes \phi_{2i}\;\hbox{ with } \phi_{1i}\in\cA_1\hbox{ and } \phi_{2i}\in\cA_2 \right\}.
\]
If $M_1, N_1$ and $M_2, N_2$ are \emph{uniform} norms on $\cA_1$ and $\cA_2$ respectively, then
so are the projective tensor product norms $M_1 \otimes M_2$ and $N_1 \otimes N_2$, and the following useful lemma holds.

\begin{lemma}
\label{lemma:tensorprodpreceq}
If $M_1 \preceq N_1$ and $M_2 \preceq N_2$, then $M_1 \otimes M_2 \preceq N_1 \otimes N_2$.
\end{lemma}

\begin{proof}
Pick $\phi, \psi \in \cA_1 \otimes \cA_2$ with $(N_1 \otimes N_2)(\phi) = (N_1 \otimes N_2)(\psi) = 1$. Fix $\eps > 0$.
By the definition of the maximal cross-norm, we can find finite collections $a_k, c_l \in \cA_1$ and $b_k, d_l \in \cA_2$ such that
\[
\phi = \sum_k a_k \otimes b_k \qand 1+\eps \geq \sum_k N_1(a_k) N_2(b_k),
\] 
and
\[
\psi = \sum_l c_l \otimes d_l \qand 1+\eps \geq \sum_l N_1(c_l) N_2(d_l).
\] 
We see that
\[
1 = (N_1 \otimes N_2)(\phi) \, (N_1 \otimes N_2)(\psi)
\geq 
\sum_{k,l} N_1(a_k) N_1(c_l) N_2(b_k) N_2(d_l) - 2\eps - \eps^2.
\]
Since $M_1 \preceq N_1$ and $M_2 \preceq N_2$, the double sum above is bounded (up to a constant) from 
below by
\[
\sum_{k,l} M_1(a_k c_l) M_2(b_k d_l)
\]
which (by definition of $M_1 \otimes M_2$ as an infimum) is further bounded from below by
\[
(M_1 \otimes M_2)\left( \left(\sum_k a_k \otimes b_k\right) \cdot \left(\sum_l c_l \otimes d_l\right) \right) = (M_1 \otimes M_2)(\phi\cdot \psi).
\]
Since $\eps > 0$ is arbitrary, we conclude that there exists a constant $C>0$ such that 
\[
(M_1 \otimes M_2)(\phi\cdot \psi) \leq C,
\]
whenever $(N_1 \otimes N_2)(\phi) = (N_1 \otimes N_2)(\psi) = 1$, which finishes the proof.
\end{proof}

\section{General coupling estimates}
\label{PrfABCDT}

In this section we prove a general estimate for measures supported on product spaces
(Proposition \ref{mainest}) which includes Proposition \ref{mainineq} 
from Section \ref{outline_main1_0} as a special case.
This result  has been used in the proofs in Sections
\ref{outline_main1_0} and \ref{sec:salg}.
We work in a more abstract setting which we now introduce. 

\subsection{Notation}

Let $(X_1,\rho_1)$ and $(X_2,\rho_2)$ be locally compact metric spaces, and fix two subalgebras 
\[
\cA_1 \subset\cC_c(X_1) \qand \cA_2 \subset\cC_c(X_2)
\] 
of Lipschitz continuous functions 
on $(X_1,\rho_1)$ and $(X_2,\rho_2)$ respectively. Let $(M_1,N_1)$ and $(M_2,N_2)$ be two pairs of 
uniform norms on $\cA_1$ and $\cA_2$ such that 
\begin{equation}
\label{ass:norms}
M_i \preceq N_i \qand \Lip_{\rho_i}(\phi_i) \ll M_i(\phi_i), \quad\quad \textrm{for all $\phi_i \in \cA_i$},
\end{equation}
for $i = 1,2$, where $\Lip_{\rho_i}$ denotes the usual Lipschitz semi-norm with respect to the metric $\rho_i$. \\

Let $m_1$ and $m_2$ be Borel probability measures on $X_1$ and $X_2$ respectively, and fix a Borel 
probability measure $\eta$ on the direct product $X_1 \times X_2$, with marginals $\eta_1$ and $\eta_2$. 
Suppose that there exist jointly continuous $\bK$-actions (here $\bK$ is either $\bR$ or $\bQ_p$)
\[
h_i : \bK \times X_i \ra X_i, \quad\quad \textrm{$i = 1,2$},
\]
which preserve the measures $m_1$ and $m_2$ respectively, such that the {diagonal flow}
\begin{equation}
\label{defh0}
h(t) \cdot (x_1,x_2) = (h_1(t) \cdot x_1, h_2(t) \cdot x_2), 
\quad \quad
\textrm{for $(x_1,x_2) \in X_1 \times X_2$ and $t \in \bK$},
\end{equation}
preserves the measure $\eta$. One readily checks that $h_i$ preserves $\eta_i$, for $i = 1,2$, as well. We further
assume that the flows $h_i$ preserve the algebras $\cA_i$ for $i = 1,2$.\\

Our goal in this section is to provide an upper bound (Proposition \ref{mainest} below) on the Wasserstein
distance $\dist_{N_1 \otimes N_2}(\eta,m_1 \otimes m_2)$ in terms 
of the Wasserstein distances 
\[
\dist_{M_1}(\eta_1,m_1) \qand \dist_{M_2}(\eta_2,m_2), 
\]
under the following three assumptions on the flows $h_1$ and $h_2$
(cf. Lemmas \ref{lemma1}, \ref{lemma2}, \ref{lemma3}):

\begin{itemize}
\item (\textit{Polynomial growth w.r.t. $N_1$})
There exist constants $A \geq 1$ and $a > 0$ such that
\begin{equation}
\label{ass:polygrowth}
N_1(h_1(t) \cdot \phi_1) \leq A \max(1,|t|)^{a}\, N_1(\phi_1), 
\end{equation}
for all $\phi_1 \in \cA_1$ and $t\in \mathbb{K}$.

\item
(\textit{Polynomial rate of mixing w.r.t. $m_1$})
There exist constants $B > 0$ and $0 < w_1 \leq 1$ and $0 < b < 1/2$ such that
\begin{equation}
\label{ass:polmix}
| m_1((h_1(t) \cdot \phi_1) \phi_1) - m_1(\phi_1)^2 | \leq B \max(1,w_1 |t|)^{-b}\, N_1(\phi_1)^2, 
\end{equation}
and for all $\phi_1 \in \cA_1$ and $t\in\mathbb{K}$.

\item (\textit{Lipschitz continuity for $h_2$})
There exist constants $C > 0$ and $0 < w_2 \leq 1$ such that
\begin{equation}
\label{ass:Holder}
\|h_2(t) \cdot \phi_2 - \phi_2\|_\infty \leq C\, w_2 |t| \, N_2(\phi_2),
\end{equation}
for all $\phi_2 \in \cA_2$ and $t\in \mathbb{K}$ satisfying $|t|\le w_2^{-1}$.

\end{itemize}

\begin{remark}
We recall (upon retaining the notation from Section \ref{sec:p2}) that we have indeed 
verified these assumptions for the flows $h_I$ and $h_J$
defined in \eqref{eq:hI} and \eqref{eq:hJ}, with respect to the Sobolev norms 
$N_1 = \cS_{d+r,I}$ and $N_2 = \cS_{d+r,J}$ for sufficiently large $d$ and $r$ such that Properties N1--N4 hold
(see Lemmas \ref{lemma1}, \ref{lemma2}, and \ref{lemma3}). 
These assumptions also holds for the flows $h_I$ and $h_J$
appearing in Section \ref{sec:hij}.
\end{remark}

\subsection{The main estimate}

The following result generalizes Proposition \ref{mainineq}. Its proof will occupy the rest of this section.

\begin{proposition}
\label{mainest}
With the notation and assumptions above, we have 
\begin{equation}
\label{mainestineq}
\dist_{N_1 \otimes N_2}(\eta,m_1 \otimes m_2)
\ll
\max\big((\mathcal{M}|V(T)|^a)^{1/2},(w_1 T)^{-b/2}, w_2 T \big),
\end{equation}
for all $T\in [w_1^{-1},w_2^{-1}]$, where
\begin{align*}
V(T)&:=\{t\in \bK:\, |t|\le T\},\\
\mathcal{M} &:= \max\big( \dist_{M_1}(\eta_1,m_1), \dist_{M_2}(\eta_2,m_2) \big),
\end{align*}
and the implied constant depends only the constants $A, B$ and $C$, and on the norms $M_1, M_2,N_1$ and $N_2$. 
In particular, the bound \eqref{mainestineq} is uniform over all couplings $\eta$ of $\eta_1$ and $\eta_2$ which are invariant under the flow $h$.
\end{proposition}

We retain the notation from the beginning of this section, and assume that \eqref{ass:polygrowth}, \eqref{ass:polmix} and \eqref{ass:Holder} hold. In particular, the letters $A,a,B,b,C$ and $w_1, w_2$ have been 
assigned fixed meanings. \\

For $T>0$, we define the linear, positive and unital operator $P_T : \cC_c(X_1) \ra \cC_c(X_1)$ by
\[
(P_T \phi_1)(x_1) := \frac{1}{|V(T)|} \int_{V(T)} (h_1(t) \cdot \phi_1)(x_1) \, dt, \quad\quad \textrm{for $\phi_1 \in\cC_c(X_1)$},
\]
where $V(T)=\{t\in \bK:\, |t|\le T\}$.
We note that $P_T$ preserves $\cC_c(X_1)$, however we stress that it may \emph{not} preserve the subalgebra $\cA_1$. 
If we denote by $P_T^*$ its adjoint on $\cC_c(X_1)^*$, then 
since the flow $h_1$ preserves the measures $\eta_1$ and $m_1$, namely,
\[
P_T^*\eta_1 = \eta_1 \qand P_T^* m_1 = m_1.
\] 
The  triangle inequality for $\dist_{N_1 \otimes N_2}$ now yields
\begin{align}
\dist_{N_1 \otimes N_2}(\eta,m_1 \otimes m_2)
\leq\; &
\dist_{N_1 \otimes N_2}(\eta,(P_T \otimes \id)^*\eta) \tag{I} \\
&+
\dist_{N_1 \otimes N_2}((P_T \otimes \id)^*\eta,\eta_1 \otimes \eta_2)  \tag{II} \\
&+
 \dist_{N_1 \otimes N_2}(\eta_1 \otimes \eta_2,m_1 \otimes m_2). \tag{III} 
\end{align}
In what follows, we shall provide bounds on each term. These bounds will readily combine to 
the bound which is asserted in Proposition \ref{mainest}.
 
\subsection{Estimating Term (I)} 

\begin{lemma}
\label{lemmaI}
For all $T \in (0,w_2^{-1}]$, we have
\[
\dist_{N_1 \otimes N_2}(\eta,(P_T \otimes \id)^*\eta) \ll C\, w_{2}T,
\]
where the implied constant depends only on the norm $N_1$. 
\end{lemma}

\begin{proof}
Pick $\phi \in \cA_1 \otimes \cA_2$, and write it as a finite sum of the form
\begin{equation}
\label{pickf}
\phi = \sum_i \phi_{1i} \otimes \phi_{2i},
\end{equation}
for some $\phi_{1i} \in \cA_1$ and $\phi_{2i} \in \cA_2$. Since the flow $h$ (defined in \eqref{defh0}) 
preserves the measure $\eta$, we have
\[
\eta(\phi) = \sum_i \eta\left( \frac{1}{|V(T)|} \int_{V(T)} (h_1(t) \cdot \phi_{1i}) \otimes (h_2(t) \cdot \phi_{2i}) \, dt \right),
\]
for all $T > 0$. Note that
\[
\eta((P_T \otimes \id)\phi) = \sum_i \eta\left( \frac{1}{|V(T)|} \int_{V(T)} (h_1(t) \cdot \phi_{1i}) \otimes \phi_{2i} \, dt \right),
\]
and thus we see that 
\[
\eta(\phi) - \eta((P_T \otimes \id)\phi) = \sum_{i} \eta\left( 
\frac{1}{|V(T)|} \int_{V(T)} (h_1(t) \cdot \phi_{1i}) \otimes (h_2(t) \cdot \phi_{2i} - \phi_{2i}) \, dt \right).
\]
Hence,
\[
| \eta(\phi) - \eta((P_T \otimes \id)\phi) | \leq \sum_i \|\phi_{1i}\|_\infty \cdot \max_{|t|\le T} \|h_2(t) \cdot \phi_{2i} - \phi_{2i}\|_\infty.
\]
Since $N_1$ is assumed to be uniform, we have $\|\phi_{1i}\|_\infty \ll N_1(\phi_{1i})$ for all $i$, and 
by \eqref{ass:Holder} we have
\[
\max_{|t|\le T} \|h_2(t) \cdot \phi_{2i} - \phi_{2i}\|_\infty \leq C\, w_{2} T\, N_2(\phi_{2i}),
\]
for every $i$. Thus, we conclude that
\[
| \eta(\phi) - \eta((P_T \otimes \id)\phi) | \ll C\,w_2 T\,  \left(\sum_i N_1(\phi_{1i}) N_2(\phi_{2i}) \right),
\]
where the implied constant depends only on $N_1$. Upon taking the infimum over all representations of $\phi$ as a finite sum as in \eqref{pickf}, we see that 
\[
| \eta(\phi) - \eta((P_T \otimes \id)\phi) | \ll C\,w_2T\, (N_1 \otimes N_2)(\phi),
\]
which finishes the proof.
\end{proof}

\subsection{Estimating Term (II)} 

\begin{lemma}
\label{lemmaII}
For all $T \ge w_1^{-1}$, we have
\[
\dist_{N_1 \otimes N_2}((P_T \otimes \id)^*\eta,\eta_1 \otimes \eta_2)
\ll
\max\big(\sqrt{A}\, |V(T)|^{a/2} \, \dist_{M_1}(\eta_1,m_1)^{1/2}, B\, (w_1T)^{-b/2}\big),
\]
where the implied constant depends only on the norms $M_1$, $N_1$ and $N_2$.
\end{lemma}

\medskip

The proof of this lemma will require new notation. Given a Borel probability measure $\nu$ on $X_1$, which is assumed to be invariant under $h_1$, we define
\[
E_T(\nu): = \sup\left\{ \left( \int_{X_1} |P_T \phi_1 - \nu(\phi_1)|^2 \, d\nu \right)^{1/2}  : \,\, \phi_1 \in \cA_1\;\hbox{ with } N_1(\phi_1) \leq 1 \right\},
\]
for $T > 0$. This expression can also be written in a more convenient form as follows.
Given a function $\phi_1 \in \cA_1$ and a $h_1$-invariant Borel probability measure $\nu$ on $X_1$, 
we define
\[
C_{\nu,\phi_1}(t) = \nu((h_1(t) \cdot \phi_1) \phi_1) - \nu(\phi_1)^2 \quad\quad \textrm{for $t \in \bK$}.
\]
Upon expanding $E_T(\nu)^2$, using that $\nu$ is $h_1$-invariant, one readily sees that
\begin{equation}
\label{expand}
E_T(\nu)^2 = \sup\left\{ \frac{1}{|V(T)|^2} \int_{V(T)} \int_{V(T)}\cC_{\nu,\phi_1}(s-t) \, ds dt  : \,\, \phi_1 \in \cA_1\;\hbox{ with } N_1(\phi_1) \leq 1 \right\}.
\end{equation}

Lemma \ref{lemmaII} is an immediate consequence of the Lemmas \ref{lemmaII.1},
\ref{lemmaII.2}, and \ref{lemmaII.3} that we now prove. 
\begin{lemma}
\label{lemmaII.1}
For all $T > 0$, we have
\begin{equation}
\label{eqlemmaII.1}
\dist_{N_1 \otimes N_2}((P_T\otimes \id)^*\eta,\eta_1 \otimes \eta_2)
\ll 
E_T(\eta_1),
\end{equation}
where the implied constant depends only on the norm $N_2$.
\end{lemma}

\begin{proof}
Pick $\phi \in \cA_1 \otimes A_2$ and write it as a finite sum of the form
\begin{equation}\label{eq:ppp}
\phi = \sum_{i} \phi_{1i} \otimes \phi_{2i},
\end{equation}
for some $\phi_{1i} \in \cA_1$ and $\phi_{2i} \in \cA_2$. By definition, 
\[
\eta(\psi \otimes 1) = \eta_1(\psi), \quad\quad \textrm{for all $\psi \in \cA_1$}.
\]
We obtain
\begin{eqnarray*}
\big| \eta((P_T \otimes \id)\phi) - (\eta_1 \otimes \eta_2)(\phi) \big| 
&\leq& 
\sum_{i} \eta\left(| \, P_T \phi_{1i} - \eta_1(\phi_{1i}) | \otimes |\phi_{2i}|\right) \\
&\leq& 
\sum_{i} \eta\left(| \, P_T \phi_{1i} - \eta_1(\phi_{1i}) | \otimes \|\phi_{2i}\|_\infty\right) \\
&=&
\sum_{i} \eta_1\left(| \, P_T \phi_{1i} - \eta_1(\phi_{1i}) |\right) \, \| \phi_{2i} \|_\infty \\
&\le &
\sum_{i} \eta_1\left(| \, P_T \phi_{1i} - \eta_1(\phi_{1i}) |^2\right)^{1/2} \, \|\phi_{2i}\|_\infty,
\end{eqnarray*}
where we used H\"older's inequality termwise.
Hence, since $N_2$ is a uniform norm, 
\begin{eqnarray*}
\big| \eta((P_T \otimes \id)\phi) - (\eta_1 \otimes \eta_2)(\phi) \big| 
&\ll &
\sum_{i} \eta_1\left(| \, P_T \phi_{1i} - \eta_1(\phi_{1i}) |^2\right)^{1/2} \, N_2(\phi_{2i}) \\
&\leq &
E_T(\eta_1) \, \left( \sum_{i} N_1(\phi_{1i}) \, N_2(\phi_{2i})\right),
\end{eqnarray*}
where the implied constant depends only $N_2$.
Thus, taking the infimum over all representations of $\phi$ as a finite sum as in 
\eqref{eq:ppp}, we get
\[
\big| \eta((P_T \otimes \id)\phi) - (\eta_1 \otimes \eta_2)(\phi) \big| 
\ll E_T(\eta_1) \, (N_1 \otimes N_2)(\phi).
\]
This completes the proof.
\end{proof}

\begin{lemma}
\label{lemmaII.2}
For all $T \geq 1$, we have
\begin{equation}
\label{eqlemmaII.1}
|E_T(\eta_1) - E_T(m_1)| \ll \sqrt{A}\, |V(T)|^{a/2} \dist_{M_1}(\eta_1,m_1)^{1/2},
\end{equation}
where the implied constant depends only on the norms $M_1$ and $N_1$.
\end{lemma}

\begin{remark}
Note the change of norms in Lemma \ref{lemmaII.2}: The expression $E_T(\cdot)$ is defined 
using the norm $N_1$ on $\cA_1$, while the asserted bound  is in terms of 
the Wasserstein distance measured with respect to the norm $M_1$. This is the 
only place in our argument where it is necessary to change the norm.
\end{remark}

\begin{proof}
Since $t-s \leq \sqrt{t^2 - s^2}$ for all $0 < s < t$, we have
\begin{equation}
\label{sqbnd}
|E_T(\eta_1) - E_T(m_1)| \leq \sqrt{|E_T(\eta)^2 - E_T(m_1)^2|}.
\end{equation}
Fix $\eps > 0$ and $T \geq 1$. We use \eqref{expand}, and pick $\phi_1 \in \cA_1$ with $N_1(\phi_1) \leq 1$ such that
\[
E_T(\eta_1)^2 \leq \frac{1}{|V(T)|^2} \int_{V(T)} \int_{V(T)}\cC_{\eta_1,\phi_1}(s-t) \, ds dt + \eps.
\]
We note that
\[
E_T(\eta_1)^2 - E_T(m_1)^2 \leq \frac{1}{|V(T)|^2} \int_{V(T)} \int_{V(T)} \big(\cC_{\eta_1,\phi_1}(s-t) -\cC_{m_1,\phi_1}(s-t) \big) \, ds dt + \eps.
\]
By definition, 
\begin{align}
\cC_{\eta_1,\phi_1}(u) -\cC_{m_1,\phi_1}(u)
=&
\big( \eta_1((h_1(u) \cdot \phi_1) \phi_1) - m_1((h_1(u) \cdot \phi_1) \phi_1) \big) \tag{I} \\
&+
\big( m_1(\phi_1)^2 - \eta_1(\phi_1)^2 \big) \tag{II}.
\end{align}
We bound the terms (I) and (II) separately.
Since $N_1(\phi_1) \leq 1$ and $N_1$ is uniform, we have
\begin{align*}
\textrm{(II)}&= (m_1(\phi_1) - \eta_1(\phi_1))(m_1(\phi_1)+ \eta_1(\phi_1))\\
&\ll \dist_{N_1}(\eta_1,m_1)\, 2\|\phi_1\|_\infty \ll \dist_{N_1}(\eta_1,m_1),
\end{align*}
where the implicit constant depends only on $N_1$.
Then since $M_1 \preceq N_1$, by Lemma \ref{lemma:normbnds},
\begin{equation}
\label{eq:II}
\textrm{(II)} \ll \dist_{M_1}(\eta_1,m_1)
\end{equation}
with the implicit constant depending only on $M_1$ and $N_1$.

To estimate the term (I), we recall that by assumption \eqref{ass:polygrowth}, 
\[
N_1(h_1(u) \cdot \phi_1) \le A \max(1,|u|)^a\, N_1(\phi_1), \quad \textrm{for all $u\in\mathbb{K}$},
\]
Hence, since $M_1 \preceq N_1$ and $N_1(\phi_1) \leq 1$, we conclude that for all $u\in\mathbb{K}$,
\begin{align}\label{eq:I}
\textrm{(I)} &\leq \dist_{M_1}(\eta_1,m_1) \, M_1((h_1(u)\cdot  \phi_1)\phi_1) \ll 
\dist_{M_1}(\eta_1,m_1)\,N_1(h_1(u) \cdot \phi_1) N_1(\phi_1) \\
&\le A\, \max(1,|u|)^a \, \dist_{M_1}(\eta_1,m_1),\nonumber
\end{align}
where the implied constant depends only on $M_1$ and $N_1$.
We use that by Lemma \ref{l:integral} below
$$
\int_{V(T)}\int_{V(T)} \max(1,|s-t|)^a\, dsdt\ll |V(T)|^{a+2}
$$
for all $T\ge 1$. Hence, combining the above estimates for (I) and (II),
we deduce that for all $T \geq 1$ (since $A\ge 1$),
\[
E_T(\eta_1)^2 - E_T(m_1)^2 \ll A\, |V(T)|^a \, \dist_{M_1}(\eta_1,m_1) + \eps,
\]
where the implicit constant depend only on $M_1$ and $N_1$. Since $\eps > 0$ is arbitrary, we can 
neglect it and conclude that
\[
E_T(\eta_1)^2 - E_T(m_1)^2 \ll A\, |V(T)|^{a} \dist_{M_1}(\eta_1,m_1).
\]
The same argument can be made with roles of $\eta_1$ and $m_1$ interchanged.
Hence,
\[
|E_T(\eta_1)^2 - E_T(m_1)^2| \ll A\, |V(T)|^{a} \dist_{M_1}(\eta_1,m_1).
\]
Now the lemma follows from \eqref{sqbnd}.
\end{proof}

\begin{lemma}
	\label{l:integral}
	Let $T\ge 1$ if $\bK=\bR$ and $T\ge 1/p$ if $\bK=\bQ_p$. Then
	$$
	I(T):=\int_{V(T)}\int_{V(T)} \max(1,|s-t|)^a\, dsdt\ll |V(T)|^{a+2}
	$$
	uniformly over $a>-1/2$.
\end{lemma}

\begin{proof}
First, we consider the case when $\bK=\bR$. By a standard change of variables,
$$
I(T)=\int_0^{2T} (2T -u) \max(1,u)^a \, du
$$
A direct computation shows that for all $T\ge 1$,
$$
\int_0^{2T} (2T -u) \max(1,u)^a \, du =
\frac{(2T)^{a+2}}{(a+1)(a+2)}-\frac{2aT}{a+1}-\frac{a}{2(a+2)}.
$$
Hence, $I(T)=O(|V(T)|^{a+2})$ uniformly over $a>-1/2$.

When $\bK=\bQ_p$, we obtain
\begin{align*}
I(p^n)&=\sum_{s,t=0}^{p^n-1}\max(1, p^n|s-t|)^a
=p^n\sum_{u=0}^{p^n-1} \max(1, p^n|u|)^a\\
&= p^n ( (p^n-p^{n-1}) p^{na}+(p^{n-1}-p^{n-2}) p^{(n-1)a}+\cdots+(p-1)p^a+1)
\ll (p^n)^{a+2}\\
&=|V(p^n)|^{a+2}.
\end{align*}
This proves the lemma.
\end{proof}

\begin{lemma}
\label{lemmaII.3}
For all $T \ge w_1^{-1}$, we have
\begin{equation}
\label{eqlemmaII.1}
E_T(m_1) \ll \sqrt{B}\, (w_1 T)^{-b/2},
\end{equation}
where the implied constant is uniform.
\end{lemma}

\begin{proof}
In view of \eqref{expand}, it suffices to show that for every $\phi_1 \in \cA_1$ with $N_1(\phi_1) \leq 1$ and $T\ge w_1^{-1}$,
we have
\begin{equation}
\label{eq:claim}
\frac{1}{|V(T)|^2} \int_{V(T)} \int_{V(T)}\cC_{m_1,\phi_1}(s-t) \, ds dt \ll B\, (w_1 T)^{-b}.
\end{equation}
By assumption \eqref{ass:polmix}, for all $u\in \mathbb{K}$,
$$
|C_{m_1,\phi_1}(u)|\le B\, \max(1,w_1|u|)^{-b}\, N_1(\phi_1)^2,
$$
so that we obtain
\[
\int_{V(T)} \int_{V(T)}\cC_{m_1,\phi_1}(s-t) \, ds dt
\le
B\, \int_{V(T)} \int_{V(T)}\max(1,w_1 |s-t|)^{-b} \, dsdt.
\]

When $\bK=\bR$, a simple change of variables gives
$$
\int_{V(T)} \int_{V(T)}\max(1,w_1 |s-t|)^{-b} \, dsdt
=w_1^{-2}\int_{V(w_1T)} \int_{V(w_1T)}\max(1,|s-t|)^{-b} \, dsdt,
$$
and estimate \eqref{eq:claim} follows directly from Lemma \ref{l:integral}.

When $\bK=\bQ_p$, we pick $i\ge 1$ such that $p^{-i}\le w_1 \le p^{-i+1}$, and observe that
\begin{align*}
\int_{V(T)} \int_{V(T)}\max(1,w_1 |s-t|)^{-b} \, dsdt
&\le
\int_{V(T)} \int_{V(T)}\max(1,p^{-i} |s-t|)^{-b} \, dsdt\\
&=p^{2i}\int_{V(Tp^{-i})} \int_{V(Tp^{-i})}\max(1,|s-t|)^{-b} \, dsdt.
\end{align*}
Hence, \eqref{eq:claim} follows from Lemma \ref{l:integral}.
\end{proof}

\subsection{Estimating Term (III)} 

\begin{lemma}
\label{lemmaIII}
For all Borel probability measures $\eta_1$ and $m_1$ on $X_1$ and $\eta_2$ and $m_2$ on $X_2$,
we have 
\[
\dist_{N_1 \otimes N_2}(\eta_1 \otimes \eta_2,m_1 \otimes m_2)
\ll 
\max(\dist_{N_1}(\eta_1,m_1), \dist_{N_2}(\eta_2,m_2)),
\]
where the implied constant depends only on the norms $N_1$ and $N_2$.
\end{lemma}

\begin{proof}
Pick $\phi \in \cA_1 \otimes \cA_2$, and write it as a finite sum of the form
\[
\phi = \sum_i \phi_{1i} \otimes \phi_{2i},
\]
for some $\phi_{1i} \in \cA_1$ and $\phi_{2i} \in \cA_2$. We note that
\[
(\eta_1 \otimes \eta_2)(\phi) - (m_1 \otimes m_2)(\phi)
=
\sum_i \big( \eta_1(\phi_{1i})\eta_2(\phi_{2i}) - m_1(\phi_{1i}) m_2(\phi_{2i}) \big).
\]
Each term in this sum can be written as
\[
(\eta_1(\phi_{1i}) - m_1(\phi_{1i})) \eta_2(\phi_{2i}) + m_1(\phi_{1i}) (\eta_2(\phi_{2i}) - m_2(\phi_{2i})),
\]
and thus its absolute value can be estimated from above by
\[
\dist_{N_1}(\eta_1,m_1) N_1(\phi_{1i})\, \|\phi_{2i}\|_\infty + \|\phi_{1i}\|_\infty\,\dist_{N_2}(\eta_2,m_2)  N_2(\phi_{2i}).
\]
Since $N_1$ and $N_2$ are uniform norms, we conclude that
\[
|(\eta_1 \otimes \eta_2)(\phi) - (m_1 \otimes m_2)(\phi)|
\ll 
\max(\dist_{N_1}(\eta_1,m_1), \dist_{N_2}(\eta_2,m_2)) \sum_i N_1(\phi_{1i})N_2(\phi_{2i}),
\]
where the implied constant depends only on $N_1$ and $N_2$. 
Hence,
\[
|(\eta_1 \otimes \eta_2)(\phi) - (m_1 \otimes m_2)(\phi)|
\ll 
 \max(\dist_{N_1}(\eta_1,m_1), \dist_{N_2}(\eta_2,m_2)) (N_1\otimes N_2)(\phi).
\]
Since $\phi$ is arbitrary, this implies the lemma.
\end{proof}

From Lemma \ref{lemmaIII}, we also deduce

\begin{corollary}
	\label{c_lemmaIII}
	For all Borel probability measures $\eta_1$ and $m_1$ on $X_1$ and $\eta_2$ and $m_2$ on $X_2$,
	we have 
	\[
	\dist_{N_1 \otimes N_2}(\eta_1 \otimes \eta_2,m_1 \otimes m_2)
	\ll 
	\max(\dist_{M_1}(\eta_1,m_1)^{1/2}, \dist_{M_2}(\eta_2,m_2)^{1/2}),
	\]
	where the implied constant depends on the norms $M_1,N_1,M_2$ and $N_2$.
\end{corollary}

\begin{proof}
Since $M_1 \preceq N_1$ and $M_2 \preceq N_2$, there are constants $E_1,E_2>0$
such that the bounds $M_i \leq E_i\, N_i$ hold on $\cA_i$, $i=1,2$. By Lemma \ref{lemma:normbnds}, we have
\[
\dist_{N_i}(\eta_i,m_i) \leq E_i \dist_{M_i}(\eta_i,m_i), \quad \textrm{for $i = 1,2$}.
\]
Hence, it follows from Lemma \ref{lemmaIII} that 
\begin{equation}
\label{eq:est}
\dist_{N_1 \otimes N_2}(\eta_1 \otimes \eta_2,m_1 \otimes m_2)
\ll
\max(\dist_{M_1}(\eta_1,m_1), \dist_{M_2}(\eta_2,m_2)).
\end{equation}
Finally, note that since $M_1$ and $M_2$ are uniform norms, there are constants $F_1, F_2 > 0$ such that for all $\phi_i\in \cA_i$,
\[
\|\phi_i\|_\infty \leq F_i\, M_i(\phi_i), \quad \textrm{for $i = 1,2$},
\]
and thus $\dist_{M_i}(\cdot,\cdot) \leq 2F_i$ on $\cP(X_i)$ for $i = 1,2$. Hence, \[
\dist_{M_i}(\cdot,\cdot) \leq \sqrt{2F_i} \dist_{M_i}(\cdot,\cdot)^{1/2},
\]
on $\cP(X_i)$ for $i = 1,2$.
These estimates combined with \eqref{eq:est} finish the proof.
\end{proof}

\subsection{Completion of the proof of Proposition \ref{mainest}}

Combining the bounds on 
Term I, Term II and Term III, which Lemma \ref{lemmaI}, Lemma \ref{lemmaII} and Corollary \ref{c_lemmaIII} 
respectively provide, we deduce that for all $T\in [w_1^{-1},w_2^{-1}]$,
\[
\dist_{N_1 \otimes N_2}(\eta,m_1 \otimes m_2)
\ll
\max(w_2 T, |V(T)|^{a/2} \dist_{M_1}(\eta_1,m_1)^{1/2},(w_1 T)^{-b/2},\dist_{M_2}(\eta_2,m_2)^{1/2}),
\]
where the implied constant depends only on the constants $A, B, C$ and on the norms $M_1,N_1,M_2$ and $N_2$. If we set
\[
\mathcal{M} = \max(\dist_{M_1}(\eta_1,m_1),\dist_{M_2}(\eta_2,m_2)),
\]
then we can merge the $\dist_{M_i}$-terms above (provided that $T \geq 1$) into $\mathcal{M}$, and thus
\[
\dist_{N_1 \otimes N_2}(\eta,m_1 \otimes m_2)
\ll
\max(w_{2} T, |V(T)|^{a/2} \sqrt{\mathcal{M}},(w_1 T)^{-b/2}),
\]
for all $T \in [w_1^{-1},w_2^{-1}]$. This finishes the proof of Proposition \ref{mainest}.

\end{document}